\documentclass[a4paper,10pt]{article}
\usepackage[bbgreekl]{mathbbol}
\usepackage{mathrsfs}
\usepackage{graphicx}
\usepackage{amsmath}
\usepackage{amsfonts}
\usepackage{amssymb}
\usepackage{amsthm}
\usepackage{subfig}

\newtheorem{theorem}{Theorem}[section]
\newtheorem{lemma}{Lemma}[section]
\newtheorem{remark}{Remark}[section]
\newtheorem{proposition}{Proposition}[section]

\begin{document}
\title{Numerical Analysis of Finite Dimensional Approximations of Kohn-Sham Models
\thanks{{This work was partially
supported by the National Science Foundation of China under grants
10871198 and 10971059, the Funds for Creative Research Groups of
China under grant 11021101, and the National Basic Research Program
of China under grant 2011CB309703.}}}
\author{Huajie Chen\thanks{LSEC, Institute of Computational Mathematics
and Scientific/Engineering Computing, Academy of Mathematics and
Systems Science, Chinese Academy of Sciences, Beijing 100190, China
({hjchen@lsec.cc.ac.cn}).} \and Xingao Gong\thanks{ Department of
Physics, Fudan University, Shanghai 200433, China
(xggong@fudan.edu.cn).} \and  Lianhua He\thanks{LSEC, Institute of
Computational Mathematics and Scientific/Engineering Computing,
Academy of Mathematics and Systems Science, Chinese Academy of
Sciences, Beijing 100190, China ({helh@lsec.cc.ac.cn}).} \and Zhang
Yang\thanks{LSEC, Institute of Computational Mathematics and
Scientific/Engineering Computing, Academy of Mathematics and Systems
Science, Chinese Academy of Sciences, Beijing 100190, China
({zyang@lsec.cc.ac.cn}).} \and Aihui Zhou\thanks{LSEC, Institute of
Computational Mathematics and Scientific/Engineering Computing,
Academy of Mathematics and Systems Science, Chinese Academy of
Sciences, Beijing 100190, China ({azhou@lsec.cc.ac.cn}).}}
\date{}
\maketitle

\begin{abstract}
In this paper, we study finite dimensional approximations of
Kohn-Sham models, which are widely used in electronic structure
calculations. We prove the convergence of the finite dimensional
approximations and derive the a priori error estimates for ground
state energies and solutions. We also provide numerical simulations
for several molecular systems that support our theory.
\end{abstract}\vskip 0.2cm

 \vskip
0.2cm


\section{Introduction}\label{sec-introduction}\setcounter{equation}{0}
Density functional theory (DFT) is a theory of many-body systems and has become a primary tool for electronic structure calculations in atoms, molecules,
and condensed matter \cite{hohenberg-kohn64,kohn-sham65,bris03,martin04,parr-yang94,saad10}. The most widely used is the Kohn-Sham  model, in which a
many-body problem of interacting electrons in a static external potential is reduced to a tractable problem of non-interacting electrons moving in an
effective potential.  The purpose of this paper is to analyze the finite dimensional approximations of Kohn-Sham models so as to provide a mathematical
justification for both the directly minimizing energy functional method \cite{payne92,schneider09} and the variational optimization method (i.e. solving
the Kohn-Sham equation self-consistently) \cite{martin04} and some understanding of several existing approximate methods in modern electronic  structure
calculations.

Throughout this paper, we restrict our mathematical analysis and numerical simulations to non-relativistic, spin-unpolarized models.
In the pseudopotential setting, the ground state solutions of the
Kohn-Sham model for a molecular system can be obtained by minimizing
the Kohn-Sham energy functional
\begin{eqnarray}\label{eq-energy-intro}\nonumber
E(\{\phi_i\})&=&\frac{1}{2}\sum_{i=1}^{N} \int_{\mathbb{R}^3}|\nabla\phi_i(x)|^2dx + \int_{\mathbb{R}^3}V_{loc}(x)\rho(x)dx +
\sum_{i=1}^N\int_{\mathbb{R}^3}\phi_i(x)V_{nl}\phi_i(x)dx \\
&& + \frac{1}{2}\int_{\mathbb{R}^3}\int_{\mathbb{R}^3}
\frac{\rho(x)\rho(y)}{|x-y|}dxdy + \int_{\mathbb{R}^3}
\mathcal{E}(\rho(x))dx
\end{eqnarray}
with respect to wavefunctions $\{\phi_i\}_{i=1}^{N}$ under the
orthogonality constraints
\begin{eqnarray*}
\int_{\mathbb{R}^3}\phi_i\phi_j=\delta_{ij}, \quad 1\leq i,j\leq N,
\end{eqnarray*}
where $N$ is the number of valence electrons in the system, $\rho=\sum_{i=1}^N|\phi_i(x)|^2$ is the electron density, $V_{loc}$ and $V_{nl}$ are the
local and nonlocal pseudopotential operators respectively, that treat the core electrons and the nuclei as a unit and represent the interactions on the
valence electrons \cite{martin04}, and $\mathcal{E}(\rho)$ denotes the exchange-correlation energy per unit volume in an electron gas with density
$\rho$.
The Euler-Lagrange equation corresponding to this minimization
problem is the so-called Kohn-Sham equation: find
$\lambda_i\in\mathbb{R}, ~\phi_i
\in H^1(\mathbb{R}^3)~(i=1,2,\cdots,N)$ such that
\begin{eqnarray}\label{problem-eigen-intro}
\left\{ \begin{array}{rcl} \left(-\frac{1}{2}\Delta +
V_{eff}(\{\phi_i\})\right)\phi_i &=& \lambda_i\phi_i\quad
\mbox{in}~\mathbb{R}^3,
\quad i=1,2,\cdots, N,\\[1ex] \displaystyle
\int_{\mathbb{R}^3}\phi_i\phi_j &=& \delta_{ij},
\end{array} \right.
\end{eqnarray}
where $V_{eff}(\{\phi_i\})$ is the effective potential relative to the last four terms in energy functional
\eqref{eq-energy-intro}. This is  a nonlinear integro-differential
eigenvalue problem, and \eqref{problem-eigen-intro} is often called
self-consistent field (SCF) equation as to emphasize the nonlinear
feature encoded in $V_{eff}(\{\phi_i\})$.
It is assumed in most of the simulations that the ground state solutions can be found by occupying the lowest eigenstates of Kohn-Sham equation
\eqref{problem-eigen-intro}. It is not known whether the assumption is true, but it seems to be most often the case in practice.

The main difficulties of numerical analysis for Kohn-Sham
models lie in what we have to either handle the global minimization
problems whose energy functionals may be nonconvex or  deal with
 the nonlinear eigenvalue problems whose eigenvalues may not be nondegenerate.
To our best knowledge, except for the very recent works of
Canc\`{e}s, Chakir, and Maday \cite{cancesM10} and Suryanarayana et
al \cite{sur10}, there is no any other numerical analysis for
Kohn-Sham models in the literature. We see that the numerical analysis
of Kohn-Sham models is crucial to understand the efficiency of the
numerical methods widely used in electronic structure calculations.
Under a coercivity assumption of the so-called  second order
optimality condition, \cite{cancesM10} provided numerical analysis
of plane wave approximations and showed that every ground state
solution can be approximated by plane wave solutions, and
\cite{sur10}  gave the convergence of ground state energy
approximations based on finite element discretizations only. In this
paper, we shall present a systematic analysis for a general finite
dimensional discretization and prove that all the limit points of
finite dimensional approximations are ground state solutions of the
system, and every ground state solution can be approximated by
finite dimensional solutions if the associated local isomorphism
condition is satisfied. We provide not only convergence of ground
state energy approximations but also convergence rates of both
eigenvalue and eigenfunction approximations. We point out that the
local isomorphism condition should be very mild and is indeed
satisfied if the second order optimality condition is provided.

Besides the Kohn-Sham models, there is another approach in DFT that is not so popular and is called of orbital-free DFT \cite{chen-zhou08,wang-carter00},
in which approximate functionals in terms of electron density alone are used for the kinetic energy of the non-interacting system and only the lowest
eigenvalue needs to be computed. There are several related works on its convergence analysis \cite{chen-gong09,langwallner-ortner09,zhou04,zhou07} and a
priori error estimates \cite{cancesM09,cancesM10,chen-he-zhou10}.

This paper is organized as follows. In the coming section, we give a brief overview of the Kohn-Sham models and some preparations. In Section
\ref{sec-finite-anal}, we derive the existence of a unique local discrete solution under some reasonable assumptions. In Section
\ref{sec-error-estimate}, we  prove the convergence of finite dimensional approximations of the ground state solutions with quite weak assumptions and
derive the error estimates of ground state energy, ground state eigenfunctions and eigenvalues. In Section \ref{sec-numerical}, we present some numerical
results that support our theory. Finally, we give some concluding remarks.

\section{Preliminaries} \label{sec-preliminaries}\setcounter{equation}{0}
Physically, the Kohn-Sham model is set over $\mathbb{R}^3$. But in a lot of
computations, $\mathbb{R}^3$ may be replaced by some polyhedral
bounded domain $\Omega\subset\mathbb{R}^3$, for example, a supercell
for crystal or a large enough cuboid for finite system, which is
reasonable since the  solution of \eqref{problem-eigen-intro}
always decays exponentially \cite{agmon81,hoffmann01,simon00}.
Thus we study numerical analysis of finite dimensional
approximations of Kohn-Sham equation as follows:
\begin{eqnarray}\label{kohn-sham-equation}
\left\{ \begin{array}{rcl} \left(-\frac{1}{2}\Delta +
V_{eff}(\{\phi_i\})\right)\phi_i &=& \lambda_i\phi_i\quad
\mbox{in}~\Omega,
\quad i=1,2,\cdots, N,\\[1ex]
\displaystyle
\int_{\Omega}\phi_i\phi_j &=& \delta_{ij},\quad i,j=1,2,\cdots, N
\end{array} \right.
\end{eqnarray}
with the Dirichlet boundary condition $\phi_i=0$ on $\partial\Omega$ for finite systems and
periodic boundary conditions for crystals,
where $\Omega\subset\mathbb{R}^3$ is a polyhedral bounded domain.

We shall use the standard notation for Sobolev spaces
$W^{s,p}(\Omega)$ and their associated norms and seminorms, see,
e.g., \cite{ciarlet78}. For $p=2$, we denote
$H^s(\Omega)=W^{s,2}(\Omega)$ and $H^1_0(\Omega)=\{v\in H^1(\Omega):
v|_{\partial\Omega}=0\}$, where $v|_{\partial\Omega}=0$ is
understood in the sense of trace, $\|\cdot\|_{s,\Omega}=
\|\cdot\|_{s,2,\Omega}$, and $(\cdot,\cdot)$ is the standard $L^2$
inner product. The space $Y^*$, the dual of the Banach space
$Y$, will also be used. For convenience, the symbol
$\lesssim$ will be used in this paper. The notation $A\lesssim
B$ means that $A\leq C B$ for some constant $C$ that is independent
of the mesh parameters.

Given $c_1\in \mathbb{R}$ and $p, c_2\in [0,\infty)$, we define
\begin{align*}\label{pol-notation}
\mathscr{P}(p,(c_1,c_2))=~\big\{ f:~ \exists ~a_1,a_2\in \mathbb{R}
\mbox{ such that } ~~c_1 t^{p}+a_1 \leq f(t) \leq c_2 t^{p}+a_2
\quad \forall ~t\geq0 \big\}.
\end{align*}
For $\bbkappa\in \mathbb{R}^{N\times N}$, we denote its Frobenius
norm by $|\bbkappa|$.
We consider the functional space\footnote{In fact, our theory also applies to
space $\mathcal{H}=(H^1_{\#}(\Omega))^N$, where  $\Omega$ is the
unit cell of a periodic lattice $\mathcal{R}$ of $\mathbb{R}^d$ and
$H^1_{\#}(\Omega)=\{v|_{\Omega}: v\in H^1_{loc}(\mathbb{R}^d)~
\mbox{and}~ v~ \mbox{is}~ \mathcal{R}-\mbox{periodic}\}$.}
$$
\mathcal{H}\equiv
(H_0^1(\Omega))^N=\{(\phi_1,\phi_2,\cdots,\phi_N):~\phi_i\in
H_0^1(\Omega)~(i=1,2,\cdots,N)\},
$$
which is a Hilbert space associated with the induced norm
$\displaystyle\|\Phi\|_{1,\Omega}=\left(\sum_{i=1}^{N}(\|\phi_i\|_{0,\Omega}^2+\|\nabla\phi_i\|_{0,\Omega}^2)\right)^{1/2}$
and inner product $\displaystyle(\nabla \Phi,\nabla
\Psi)=\sum_{i=1}^N(\nabla \phi_i,\nabla \psi_i)$ for
$\Phi=(\phi_1,\phi_2,\cdots,\phi_N),
\Psi=(\psi_1,\psi_2,\cdots,\psi_N)\in\mathcal{H}$.

For simplicity of notation, we will sometimes abuse the notation by
\begin{eqnarray*}\label{def-norm-abuse}
\|\Phi\|_{m,\omega}=\left(
\sum_{i=1}^{N}\|\phi_i\|_{m,\omega}^2\right)^{1/2},\quad
\|\Phi\|_{0,p,\omega}=\left(\sum_{i=1}^{N}\|\phi_i\|_{0,p,\omega}^p\right)^{1/p}
\end{eqnarray*}
for subdomain $\omega\subset\Omega$ and
$\Phi=(\phi_1,\phi_2,\cdots,\phi_N)\in\mathcal{H}$. For any
$\Phi=(\phi_1,\phi_2,\cdots,\phi_N),
\Psi=(\psi_1,\psi_2,\cdots,\psi_N)\in\mathcal{H}$, we define
$\displaystyle\rho_{\Phi}=\sum_{i=1}^{N}|\phi_i|^2$ and
$$
\Phi^T\Psi=\left( \int_{\Omega}\phi_i\psi_j \right)_{i,j=1}^N
\in\mathbb{R}^{N\times N}.
$$
In our discussion, we shall also use the following three spaces:
$$\mathcal{S}^{N\times N}=\{M\in\mathbb{R}^{N\times N}: M^T=M\},~~
\mathcal{A}^{N\times N}=\{M\in\mathbb{R}^{N\times N}: M^T=-M\},
$$
and
\begin{eqnarray*}\label{eq-class}
\mathbb{Q}=\{\Phi\in\mathcal{H}:\Phi^T\Phi=I^{N\times N}\}.
\end{eqnarray*}
We may decompose  $\mathcal{H}$ as a direct sum of three subspaces \cite{edelman98,maday00}:
\begin{eqnarray*}
\mathcal{H}=\mathcal{S}_\Phi\oplus\mathcal{A}_\Phi\oplus\mathcal{T}_{\Phi}
\end{eqnarray*}
for any $\Phi\in \mathbb{Q}$, where
$\mathcal{S}_{\Phi}=\Phi\mathcal{S}^{N\times N}$,
$\mathcal{A}_{\Phi}=\Phi\mathcal{A}^{N\times N}$, and $
\mathcal{T}_{\Phi}=\left\{\Psi\in\mathcal{H}:\Psi^T\Phi=0\in\mathbb{R}^{N\times
N}\right\}. $

\subsection{Kohn-Sham models}\label{subsec-KSenergy}

In the most commonly setting of local density approximation (LDA)
\cite{martin04}, the associated Kohn-Sham energy functional of
(\ref{kohn-sham-equation}) is expressed as
\begin{eqnarray}\label{eq-energy}
E(\Phi)= \int_{\Omega} \left(\sum_{i=1}^{N}
\frac{1}{2}|\nabla\phi_i|^2 + V_{loc}(x)\rho_{\Phi} +
\sum_{i=1}^N\phi_iV_{nl}\phi_i + \mathcal{E}(\rho_{\Phi})\right)
+\frac{1}{2} D(\rho_{\Phi},\rho_{\Phi})
\end{eqnarray}
for $\Phi=(\phi_1,\phi_2,\cdots,\phi_N)\in\mathcal{H}$, where
$V_{loc}$ is a smooth local pseudopotential, $V_{nl}$ is the
nonlocal pseudopotential operator (see, e.g., \cite{martin04}) given
by
$$
V_{nl}\phi=\sum_{j=1}^M(\phi,\zeta_j)\zeta_j
$$
with $\zeta_j\in L^2(\Omega) (j=1,2,\cdots,M)$,
 $D(\rho_{\Phi},\rho_{\Phi})$ denotes electron-electron coulomb
energy with
$$
D(f,g)=\int_{\Omega}f(g*r^{-1}) =
\int_{\Omega}\int_{\Omega}f(x)g(y)\frac{1}{|x-y|}dxdy,
$$
and $\mathcal{E}(t)$ is some real function over $[0,\infty)$.
We may assume  that $V_{loc}\in L^2(\Omega)$. We see that the
function $\mathcal{E}:[0,\infty)\rightarrow\mathbb{R}$ does not have
a simple analytical expression. In applications, we shall use some
approximations to $\mathcal{E}$, for which  we shall make the
assumption that $\mathcal{E}(t)\in \mathscr{P}(3,(c_1,c_2))$ with
$c_1\geq 0$ or $\mathcal{E}(t)\in \mathscr{P}(4/3,(c_1,c_2))$ that
is satisfied by most of the approximations.

First of all, we have
\begin{proposition}\label{proposition-invariant}
Functional \eqref{eq-energy} is invariant with respect to unitary
transformations, i.e.,
$$
E(\Phi)=E(\Phi U)\qquad \forall~\Phi\in \mathbb{Q}
$$
for any matrix $U=(u_{ij})_{i,j=1}^N\in\mathcal{O}^{N\times N}$,
where $\mathcal{O}^{N\times N}$ is the set of orthogonal matrices.
\end{proposition}

Using similar arguments in \cite{chen-gong09}, we obtain that
$E(\Psi)$ is bounded below over $\mathbb{Q}$. More precisely, we
have
\begin{proposition}\label{proposition-coercive}
There exist constants $C>0$ and $b>0$ such that
\begin{eqnarray}\label{eq-coercive}
E(\Psi)\geq C^{-1} \|\Psi\|^2_{1,\Omega} -b \quad\forall~
\Psi\in\mathbb{Q}.
\end{eqnarray}
\end{proposition}

To prove the convergence of the numerical approximations, we
 need the lower semi-continuity of the energy functional in the
weak topology of $\mathcal{H}$, whose proof can be referred to
\cite{chen-gong09}.

\begin{proposition}\label{proposition-lower-semi-continuous}
If $\Psi_k$ converge weakly to $\Psi$ in $\mathcal{H}$, then
\begin{eqnarray*}\label{eq-lsc}
E(\Psi) \leq \liminf_{k\rightarrow \infty} E(\Psi_k).
\end{eqnarray*}
\end{proposition}

The ground state  energy of the system is the global minimum of
$E(\Psi)$ in the admissible class $\mathbb{Q}$ and we shall study
the following minimization problem
\begin{eqnarray}\label{problem-min}
\inf\left\{ E(\Phi):\Phi\in \mathbb{Q} \right\}.
\end{eqnarray}
The existence of a minimizer of (\ref{problem-min}) can be found in
\cite{anantharaman09,bris93,sur10} or by similar arguments to that
in the proof of Theorem \ref{theo-convergence}. We see from
Proposition \ref{proposition-invariant} that if $\Phi$ is a
minimizer of \eqref{problem-min}, then $\Phi U \in\mathbb{Q}$ is
also a minimizer for any  $U\in \mathcal{O}^{N\times N}$. Note that
the uniqueness of a minimizer of (\ref{problem-min}) is open even up
to an orthogonal transform since the energy functional may not be
convex for almost all systems of practical interest. Therefore, we
need to define the set of ground state solutions as follows
\begin{eqnarray*}\label{set-ground}
\mathcal{G}=\left\{\Phi\in\mathbb{Q}:E(\Phi)=\min_{\Psi\in\mathbb{Q}}E(\Psi)\right\}.
\end{eqnarray*}

We see that a minimizer $\Phi=(\phi_1,\phi_2,\cdots,\phi_N)$ of
\eqref{problem-min} satisfies the associated Euler-Lagrange
equation:
\begin{eqnarray}\label{problem-eigen-compact-L}
\left\{ \begin{array}{rcl} (A_{\Phi}\phi_{i},v) &=& \displaystyle
\big( \sum_{j=1}^N \lambda_{ij}\phi_{j},v \big)
\quad\forall ~v\in H_0^1(\Omega), \quad i=1,2,\cdots,N,\\[1ex]
\displaystyle \int_{\Omega}\phi_{i}\phi_{j} &=& \delta_{ij},
\end{array} \right.
\end{eqnarray}
where $A_{\Phi}$ is the Kohn-Sham Hamiltonian operator given by
\begin{eqnarray}\label{eq-operator-A}
A_{\Phi} = -\frac{1}{2}\Delta + V_{loc} + V_{nl} +
\int_{\Omega}\frac{\rho_{\Phi}(y)}{|\cdot-y|}dy +
\mathcal{E}'(\rho_{\Phi})
\end{eqnarray}
with the Lagrange multiplier
\begin{eqnarray}\label{eq-Lambda}
\Lambda=(\lambda_{ij})_{i,j=1}^N =\left(
\int_{\Omega}\phi_jA_{\Phi}\phi_i \right)_{i,j=1}^N.
\end{eqnarray}
We define the set of ground state eigenpairs by
\begin{eqnarray*}
\Theta=\left\{(\Lambda,\Phi)\in\mathbb{R}^{N\times
N}\times\mathbb{Q}: \Phi\in \mathcal{G}~\mbox{and}~
(\Lambda,\Phi)\mbox{ solves
\eqref{problem-eigen-compact-L}}\right\}.
\end{eqnarray*}
Proposition \ref{proposition-coercive} and \eqref{eq-Lambda} imply that the ground state solutions are uniformly bounded
\begin{eqnarray}\label{eq-bounded1}
\sup_{(\Lambda,\Phi)\in \Theta}(\|\Phi\|_{1,\Omega}+|\Lambda|)<C
\end{eqnarray}
for some constant $C$.

To obtain the a priori error estimates of the finite dimensional
approximations, we shall represent Kohn-Sham equation in another
setting. Define
\begin{eqnarray*}
Y=\mathbb{R}^{N\times N}\times\mathcal{H}
\end{eqnarray*}
with the associated norm $\|(\Lambda,\Phi)\|_Y=|\Lambda|+\|\Phi\|_{1,\Omega}$ for each $(\Lambda,\Phi)\in Y$.
We may rewrite \eqref{problem-eigen-compact-L} as a nonlinear problem
as follows:
\begin{eqnarray}\label{problem-abstract}
F((\Lambda,\Phi))=0\in Y^*,
\end{eqnarray}
where $F:Y\rightarrow Y^*$ is given by
\begin{eqnarray}\label{eq-F}
\langle F((\Lambda,\Phi)),(\bbchi,\Gamma) \rangle = \sum_{i=1}^N
\big( A_{\Phi}\phi_i-\sum_{j=1}^N \lambda_{ij}\phi_{j},\gamma_i
\big) +
\sum_{i,j=1}^N\chi_{ij}\big(\int_{\Omega}\phi_i\phi_j-\delta_{ij}\big)
\end{eqnarray}
with $\Gamma=(\gamma_1,\gamma_2,\cdots,\gamma_N)\in\mathcal{H}$ and
$\bbchi=(\chi_{ij})_{i,j=1}^N\in \mathbb{R}^{N\times N}$.

The Fr\'{e}chet derivative $F'_{(\Lambda,\Phi)}$ of $F$ at
$(\Lambda,\Phi): Y\rightarrow Y^*$ is defined as
\begin{eqnarray}\label{eq-F-derivative}
& &\langle F'_{(\Lambda,\Phi)}((\bbmu,\Psi)),(\bbchi,\Gamma) \rangle\nonumber\\
&=& \langle \mathcal{L}'_{\Phi}(\Lambda,\Phi)\Psi,\Gamma \rangle -
\sum_{i,j=1}^N (\mu_{ij}\phi_{j}, \gamma_i)  + \sum_{i,j=1}^N
\chi_{ij}\int_{\Omega}(\psi_i\phi_j+\phi_i\psi_j) ~~\forall~
(\bbmu,\Psi), (\bbchi,\Gamma)\in Y,
\end{eqnarray}
where
\begin{eqnarray}\label{eq-L}
&&\langle \mathcal{L}'_{\Phi}(\Lambda,\Phi)\Psi,\Gamma \rangle
=\frac{1}{2}E''(\Phi)(\Psi,
\Gamma)-\sum_{i,j=1}^N(\lambda_{ij}\psi_j,\gamma_i)\nonumber\\
&=&\sum_{i=1}^N\big( \frac{1}{2}(\nabla\psi_i,\nabla \gamma_i) +
(V_{loc}\psi_i,\gamma_i) +
\sum_{j=1}^M(\zeta_j,\psi_i)(\zeta_j,\gamma_i) +
(\mathcal{E}'(\rho_{\Phi})\psi_i,\gamma_i)+ D(\rho_{\Phi},\psi_i
\gamma_i) \nonumber\\ &&- (\sum_{j=1}^N \lambda_{ij}\psi_{j},
\gamma_i) + \big(2\phi_i \mathcal{E}''(\rho_{\Phi})\sum_{j=1}^N
\phi_j\psi_j,\gamma_i \big) + \sum_{j=1}^N 2D(\phi_j\psi_j,\phi_i
\gamma_i)\big)
\end{eqnarray}
for $\Psi=(\psi_1,\psi_2,\cdots,\psi_N)\in\mathcal{H}$ and
$\bbmu=(\mu_{ij})_{i,j=1}^N\in \mathbb{R}^{N\times N}$.

\subsection{Basic assumptions}
The analysis of finite dimensional approximations will be carried
out under certain  assumptions, which are stated as follows
\begin{flushleft}
{\bf A1}~~ $|\mathcal{E}'(t)|+|t\mathcal{E}''(t)|\in
\mathscr{P}(p_1,(c_1,c_2))$ for some $p_1\in [0,2]$.
\\[1ex]
{\bf A2}~~ There exists a constant $\alpha\in(0,1]$ such that
$|\mathcal{E}''(t)| +|t\mathcal{E}'''(t)|\lesssim 1+t^{\alpha-1}
\quad \forall~ t>0$.
\\[1ex]
 {\bf A3}~~ If $(\Lambda,\Phi)$ is a solution of \eqref{problem-eigen-compact-L}, then
$\mathcal{L}'_{\Phi}(\Lambda,\Phi)$ is an isomorphism from
$\mathcal{T}_{\Phi}$ to $\mathcal{T}_{\Phi}$, namely, there exists a
positive constant $\gamma$ depending on $(\Lambda,\Phi)$ such that
\begin{eqnarray}\label{assumption-a2}
\inf_{\Psi\in\mathcal{T}_{\Phi}}\sup_{\Gamma\in\mathcal{T}_{\Phi}}
\frac{\langle\mathcal{L}'_{\Phi}(\Lambda,\Phi)\Psi,\Gamma\rangle}
{\|\Psi\|_{1,\Omega}\|\Gamma\|_{1,\Omega}} \geq \gamma.
\end{eqnarray}
\end{flushleft}

We see that Assumption {\bf A2} implies  Assumption {\bf A1} and the
commonly used $X_{\alpha}$ and LDA exchange-correction energy
satisfy Assumption {\bf A2} \cite{cancesM09,chen-gong09}. We shall
mention that the above assumptions are necessary for the a priori error estimate,
but none of them
will be used in our convergence analysis of finite dimensional approximations (in Section \ref{subsec-convergence}).

\begin{remark}
It is open whether Assumption {\bf A3} holds for all Kohn-Sham
models, though it may hold for semiconductors and ``closed shell"
atoms and molecules.
We see that the following assumption
\begin{eqnarray}\label{coervicty}
\langle\mathcal{L}'_{\Phi}(\Lambda,\Phi)\Psi,\Psi\rangle \geq\gamma\|\Psi\|^2_{1,\Omega} \quad \forall~\Psi\in\mathcal{T}_{\Phi},
\end{eqnarray}
which implies \eqref{assumption-a2}, is employed in  \cite{cancesM10,schneider09}.   Note that (\ref{coervicty}) is equivalent to \eqref{assumption-a2} when $(\Lambda,\Phi)$ is the ground state
solution of \eqref{problem-eigen-compact-L}.
\end{remark}

The following lemma will be used in our analysis of the local
uniqueness of discrete solution.
\begin{lemma}\label{lemma-con-holder}
Let $y_1=(\Lambda_1,\Phi_1)$ and $y_2=(\Lambda_2,\Phi_2)\in Y$
satisfy $\|y_1\|_Y+\|y_2\|_Y\leq \bar{C}$. If Assumption {\bf A1} is
satisfied, then there exists a constant $C_F$ depending on $\bar{C}$
such that
\begin{eqnarray}\label{eq-lipschitz}
\|F(y_1)-F(y_2)\|\leq C_{F}\|y_1-y_2\|_Y ~~\forall~ y_1,y_2\in Y.
\end{eqnarray}
Moreover, if Assumption {\bf A2} is satisfied, then there is a
constant $C'_{F}$ such that
\begin{eqnarray}\label{eq-holder}
\|F'_{y_1}-F'_{y_2}\|\leq
C'_{F}(\|y_1-y_2\|^{\alpha}_Y+\|y_1-y_2\|^2_Y)~~\forall~ y_1,y_2\in
Y.
\end{eqnarray}
\end{lemma}

\begin{proof}
To prove \eqref{eq-lipschitz}, it is sufficient to show that
\begin{eqnarray}\label{n-estimate}
\big( A_{\Phi_1}\Phi_1 - A_{\Phi_2}\Phi_2, \Gamma \big) \le C
\|\Phi_1-\Phi_2\|_{1,\Omega}\|\Gamma\|_{1,\Omega}\quad \forall~
\Gamma\in \mathcal{H},
\end{eqnarray}
which together with \eqref{eq-F} indeed implies
\eqref{eq-lipschitz}.
Using the H{\" o}lder inequality and the Sobolev inequality, we have
for $i=1,2,\cdots,N$ that
\begin{eqnarray*}
&&\big((-\frac{1}{2}\Delta+V_{loc})\phi_{1,i}-(-\frac{1}{2}\Delta+V_{loc})\phi_{2,i},v\big)\nonumber\\
&\leq&
\frac{1}{2}\|\phi_{1,i}-\phi_{2,i}\|_{1,\Omega}\|v\|_{1,\Omega} +
\|V_{loc}\|_{0,\Omega}\|\phi_{1,i}-\phi_{2,i}\|_{0,3,\Omega}\|v\|_{0,6,\Omega}
\\ &\lesssim& \|\phi_{1,i}-\phi_{2,i}\|_{1,\Omega}\|v\|_{1,\Omega} \quad\forall~v\in
H_0^1(\Omega)
\end{eqnarray*}
and hence
\begin{eqnarray}\label{proof-44}
\big((-\frac{1}{2}\Delta+V_{loc})\Phi_1-(-\frac{1}{2}\Delta+V_{loc})\Phi_2,\Gamma\big)
\lesssim\|\Phi_1-\Phi_2\|_{1,\Omega}\|\Gamma\|_{1,\Omega}
\quad\forall~\Gamma\in\mathcal{H}.
\end{eqnarray}
Due to
\begin{eqnarray*}
 \big(V_{nl}\Phi_1-V_{nl}\Phi_2,\Gamma\big)=\sum_{i=1}^N\big(\sum_{j=1}^M(\zeta_j,\phi_{1,i}-\phi_{2,i})\zeta_j,\gamma_i\big),
\end{eqnarray*}
we obtain
\begin{eqnarray}\label{proof-44-}
 \big(V_{nl}\Phi_1-V_{nl}\Phi_2,\Gamma\big) \lesssim
\sum_{i=1}^N\|\phi_{1,i}-\phi_{2,i}\|_{0,\Omega}\|\gamma_i\|_{0,\Omega}
~\lesssim~ \|\Phi_{1}-\Phi_2\|_{1,\Omega}\|\Gamma\|_{1,\Omega} \quad
\forall~\Gamma\in\mathcal{H}.
\end{eqnarray}

Obviously
\begin{eqnarray*}
(\mathcal{E}'(\rho_{\Phi_1})\Phi_1-\mathcal{E}'(\rho_{\Phi_2})\Phi_2,\Gamma)\lesssim
\|\Phi_1-\Phi_2\|_{1,\Omega}\|\Gamma\|_{1,\Omega}\quad
\forall~\Gamma\in \mathcal{H}
\end{eqnarray*}
when $p_1=0$ in Assumption {\bf A1}. If  Assumption {\bf A1} is
satisfied for $p_1\in (0,2]$, then there exists
 $\delta_i\in
[0,1]$ such that
\begin{eqnarray}\label{proof-45} \nonumber
(\mathcal{E}'(\rho_{\Phi_1})\Phi_1-\mathcal{E}'(\rho_{\Phi_2})\Phi_2,\Gamma)
&=& \sum_{i=1}^N\int_{\Omega} \big(
\mathcal{E}'(\rho_{\Phi_1})\phi_{1,i} -
\mathcal{E}'(\rho_{\Phi_2})\phi_{2,i} \big) \gamma_i \\\nonumber &=&
\sum_{i=1}^N\int_{\Omega}
(\mathcal{E}'(\rho_{\xi})+2\xi_i^2\mathcal{E}''(\rho_{\xi}))
(\phi_{1,i}-\phi_{2,i})\gamma_i \\\nonumber &\leq& \sum_{i=1}^N
\|\mathcal{E}'(\rho_{\xi})+2\xi_i^2\mathcal{E}''(\rho_{\xi})\|_{0,3/p_1,\Omega}
\|\phi_{1,i}-\phi_{2,i}\|_{0,6,\Omega} \|\gamma_i\|_{0,6/(5-2p_1),\Omega} \\
&\lesssim& \sum_{i=1}^N \|\rho_{\xi}\|_{0,3,\Omega}^{p_1}\|\phi_{1,i}-\phi_{2,i}\|_{1,\Omega}\|\gamma_i\|_{1,\Omega} \lesssim
\|\Phi_1-\Phi_2\|_{1,\Omega}\|\Gamma\|_{1,\Omega},
\end{eqnarray}
where $\xi=(\xi_1,\xi_2,\cdots,\xi_N)$ with
$\xi_i=\delta_i\phi_{1,i}+(1-\delta_i)\phi_{2,i}$, and the H{\"
o}lder inequality, the Sobolev inequality, and the fact
\begin{eqnarray*}
\|\rho_{\xi}\|_{0,3,\Omega}\lesssim
\|\rho_{\Phi_1}\|_{0,3,\Omega}+\|\rho_{\Phi_2}\|_{0,3,\Omega}
\lesssim\|\Phi_1\|^2_{1,\Omega}+\|\Phi_2\|^2_{1,\Omega}\le\bar{C}^2
\end{eqnarray*}
are used.

For Coulomb term, we obtain from the Young's inequality and the H{\"
o}lder inequality that
\begin{eqnarray*}
\|r^{-1}*(\rho_{\Phi_1}-\rho_{\Phi_2})\|_{0,\infty,\Omega} \lesssim
\|r^{-1}\|_{0,\tilde{\Omega}}\|\rho_{\Phi_1}-\rho_{\Phi_2}\|_{0,\Omega}
\lesssim \|r^{-1}\|_{0,\tilde{\Omega}}\|\Phi_1-\Phi_2\|_{1,\Omega},
\end{eqnarray*}
where $\tilde{\Omega} =\{x-y:x,y\in\Omega\}$.
Since
\begin{eqnarray*}\label{proof-43}\nonumber
&& \int_{\Omega} \big( (r^{-1}*\rho_{\Phi_1})\phi_{1,i} -
(r^{-1}*\rho_{\Phi_2})\phi_{2,i} \big) v\\ \nonumber &=&
\int_{\Omega} (r^{-1}*\rho_{\Phi_1})(\phi_{1,i}-\phi_{2,i})v +
\int_{\Omega} r^{-1}*(\rho_{\Phi_1}-\rho_{\Phi_2})\phi_{2,i} v \\
\nonumber &\leq& \|r^{-1}*\rho_{\Phi_1}\|_{0,\infty,\Omega}
\|\phi_{1,i}-\phi_{2,i}\|_{0,\Omega} \|v\|_{0,\Omega} +
\|r^{-1}*(\rho_{\Phi_1}-\rho_{\Phi_2})\|_{0,\infty,\Omega}
\|\phi_{2,i}\|_{0,\Omega} \|v\|_{0,\Omega}
\\ &\lesssim& \|\phi_{1,i}-\phi_{2,i}\|_{1,\Omega}\|v\|_{1,\Omega} +
\|\Phi_1-\Phi_2\|_{1,\Omega}\|v\|_{1,\Omega} \quad\forall ~v\in
H_0^1(\Omega)
\end{eqnarray*}
holds for $i=1,2,\cdots,N$, we have
\begin{eqnarray}\label{proof-46}
((r^{-1}*\rho_{\Phi_1})\Phi_1 - (r^{-1}*\rho_{\Phi_2})\Phi_2,\Gamma)
\lesssim \|\Phi_1-\Phi_2\|_{1,\Omega}\|\Gamma\|_{1,\Omega} \quad
\forall~ \Gamma\in \mathcal{H}.
\end{eqnarray}
Taking \eqref{proof-44}, \eqref{proof-44-}, \eqref{proof-45},
\eqref{proof-46} and  definition \eqref{eq-operator-A} into account,
we then arrive at (\ref{n-estimate}).

If Assumption {\bf A2} holds, then following \cite[Lemma
4.5]{cancesM10} we obtain for $\Psi=(\psi_1,\psi_2,\cdots,\psi_N),
\Gamma=(\gamma_1,\gamma_2,\cdots,\gamma_N)\in\mathcal{H}$ that
\begin{eqnarray}\label{tmp-1}
|(\mathcal{E}'(\rho_{\Phi_1})\Psi,\Gamma) -
(\mathcal{E}'(\rho_{\Phi_2})\Psi,\Gamma)|&=& \int_{\Omega}\int_0^1
2\mathcal{E}''(\rho_{\Phi(t)})(\sum_{i=1}^N\phi_i(t)(\phi_{1,i}-\phi_{2,i}))(\sum_{i=1}^N\psi_i\gamma_i)dt\nonumber\\
&\lesssim&\int_{\Omega}\int_{0}^1(1+\rho_{\Phi(t)}^{\alpha-1})
\rho_{\Phi(t)}^{1/2}\rho_{\Phi_1-\Phi_2}^{1/2}\rho_{\Psi}^{1/2}\rho_{\Gamma}^{1/2}dt\nonumber\\
\end{eqnarray}
and
\begin{eqnarray}\label{proof-73}
&& \sum_{i=1}^N(\phi_{1,i} \mathcal{E}''(\rho_{\Phi_1})\sum_{j=1}^N \phi_{1,j}\psi_j,\gamma_i) - \sum_{i=1}^N(\phi_{2,i}
\mathcal{E}''(\rho_{\Phi_2})\sum_{j=1}^N \phi_{2,j}\psi_j,\gamma_i) \nonumber\\ &=& \int_{\Omega}\int_0^1 \left[
\mathcal{E}''(\rho_{\Phi(t)})\big(\sum_{i=1}^N\phi_i(t)\psi_i\big) \big(\sum_{i=1}^N(\phi_{1,i}-\phi_{2,i})\gamma_i\big) +
\mathcal{E}''(\rho_{\Phi(t)})\big(\sum_{i=1}^N (\phi_{1,i}-\phi_{2,i})\psi_i\big) \big(\sum_{i=1}^N \phi_i(t) \gamma_i\big) \right.\nonumber\\  && +
\left. \mathcal{E}'''(\rho_{\Phi(t)})\big(\sum_{i=1}^N\phi_i(t)(\phi_{1,i}-\phi_{2,i})\big) \big(\sum_{i=1}^N\phi_i(t)\psi_i\big) \big(\sum_{i=1}^N
\phi_i(t)\gamma_i\big) \right]dt  \nonumber \\&\lesssim & \int_{\Omega}\int_0^1 (1+\rho_{\Phi(t)}^{\alpha-1})\rho_{\Phi(t)}^{1/2}
\rho_{\Phi_1-\Phi_2}^{1/2}\rho_{\Psi}^{1/2}\rho_{\Gamma}^{1/2}dt ,
\end{eqnarray}
where $\Phi(t)=\Phi_1+t(\Phi_2-\Phi_1)$ with $t\in [0,1]$.

For all $0<\alpha \leq 1/2$, we have
\begin{eqnarray*}
&&\int_0^1 \rho_{\Phi(t)}^{\alpha-1/2}dt=\int_0^1\big(\sum_{i=1}^N \phi_{1,i}^2+2t\sum_{i=1}^N \phi_{1,i}(\phi_{2,i}-\phi_{1,i})+t^2\sum_{i=1}^N
(\phi_{2,i}-\phi_{1,i})^2\big)^{\alpha-1/2}dt\nonumber\\
&=&\int_0^1 \left(\sum_{i=1}^N \phi_{1,i}^2+\sum_{i=1}^N(\phi_{2,i}-\phi_{1,i})^2\big(t+\frac{\sum_{i=1}^N
\phi_{1,i}(\phi_{2,i}-\phi_{1,i})}{\sum_{i=1}^N(\phi_{2,i}-\phi_{1,i})^2}\big)^2-\frac{\big(\sum_{i=1}^N\phi_{1,i}(\phi_{2,i}-\phi_{1,i})\big)^2}{\sum_{i=1}^N
(\phi_{2,i}-\phi_{1,i})^2}\right)^{\alpha-1/2}dt\nonumber\\
&\leq & \int_0^1 |t+\frac{\sum_{i=1}^N
\phi_{1,i}(\phi_{2,i}-\phi_{1,i})}{\sum_{i=1}^N(\phi_{2,i}-\phi_{1,i})^2}|^{2\alpha-1}\big(\sum_{i=1}^N(\phi_{2,i}-\phi_{1,i})^2\big)^{\alpha-1/2}dt\leq
\frac{1}{\alpha 2^{2\alpha}}\rho_{\Phi_1-\Phi_2}^{\alpha-1/2},
\end{eqnarray*}
which together with the fact that $0\leq \rho_{\Phi(t)}\leq 2(\rho_{\Phi_1}+t^2\rho_{\Phi_1-\Phi_2})$ implies that for all $0< \alpha \leq 1$
\begin{eqnarray}\label{proof-75}
&&\int_{\Omega}\int_0^1 (1+\rho_{\Phi(t)}^{\alpha-1})\rho_{\Phi(t)}^{1/2} \rho_{\Phi_1-\Phi_2}^{1/2}\rho_{\Psi}^{1/2}\rho_{\Gamma}^{1/2}dt\lesssim
\int_{\Omega}(\rho_{\Phi_1-\Phi_2}^{\alpha/2}+\rho_{\Phi_1-\Phi_2})\rho_{\Psi}^{1/2}\rho_{\Gamma}^{1/2}\nonumber\\
&\lesssim&\|\rho_{\Phi_1-\Phi_2}^{\alpha/2}\|_{0,6/\alpha,\Omega}\|\rho_{\Psi}^{1/2}\|_{0,12/(6-\alpha),\Omega}\|\rho_{\Gamma}^{1/2}\|_{0,12/(6-\alpha),\Omega}
+\|\rho_{\Phi_1-\Phi_2}\|_{0,3,\Omega}\|\rho_{\Psi}^{1/2}\|_{0,3,\Omega}\|\rho_{\Gamma}^{1/2}\|_{0,3,\Omega}\nonumber\\
&\lesssim & (\|\Phi_1-\Phi_2\|_{1,\Omega}^{\alpha}+\|\Phi_1-\Phi_2\|_{1,\Omega}^2)\|\Psi\|_{1,\Omega}\|\Gamma\|_{1,\Omega}
\end{eqnarray}

Similar arguments to that in \eqref{proof-46} yield that
\begin{eqnarray}\label{proof-74} \nonumber
&& \sum_{j=1}^N |D(\phi_{1,j}\psi_j,\phi_{1,i}v)-D(\phi_{2,j}\psi_j,\phi_{2,i}v)| \\
\nonumber  &\leq& \sum_{j=1}^N
|D(\phi_{1,j}\psi_j-\phi_{2,j}\psi_j,\phi_{1,i}v)| + \sum_{j=1}^N
|D(\phi_{2,j}\psi_j,\phi_{1,i}v-\phi_{2,i}v)| \\ \nonumber
&\lesssim&
\sum_{j=1}^N\|\phi_{1,j}-\phi_{2,j}\|_{1,\Omega}\|\psi_j\|_{1,\Omega}\|v\|_{1,\Omega}
+ \sum_{j=1}^N \|\phi_{1,i}-\phi_{2,i}\|_{1,\Omega}\|\psi_j\|_{1,\Omega}\|v\|_{1,\Omega} \\
&\lesssim&
\|\Phi_1-\Phi_2\|_{1,\Omega}\|\Psi\|_{1,\Omega}\|v\|_{1,\Omega}
\quad \forall~\Psi\in\mathcal{H},~\forall~v\in H_0^1(\Omega).
\end{eqnarray}

Therefore, taking \eqref{eq-F-derivative}, \eqref{eq-L}, \eqref{tmp-1}, \eqref{proof-73}, \eqref{proof-75} and \eqref{proof-74} into account, we get
\begin{eqnarray*}
\langle (F'_{y_1}-F'_{y_2}) ((\bbmu,\Psi)),(\bbchi,\Gamma) \rangle
\lesssim
(\|y_1-y_2\|^{\alpha}_Y+\|y_1-y_2\|^{2}_Y)\|(\bbmu,\Psi)\|_Y\|(\bbchi,\Gamma)\|_Y
\quad \forall~(\bbmu,\Psi),(\bbchi,\Gamma)\in Y,
\end{eqnarray*}
which implies \eqref{eq-holder} and completes the proof.

\end{proof}

\section{Finite dimensional approximations}
\label{sec-finite-anal}\setcounter{equation}{0}
 For the sake of
generality, we will not concentrate on any specific approximation,
rather we shall study approximations in a class of finite
dimensional subspaces $S_n\subset X~ (n=1,2,\cdots)$ that satisfy
\begin{eqnarray}\label{eq-approx-simple}
\lim_{n\to\infty}\inf_{\psi\in S_n}\|\psi-\phi\|_{1,\Omega}=0 \quad\forall~\phi\in X,
\end{eqnarray}
where $X$ is some Banach space containing the eigenfunctions of
(\ref{kohn-sham-equation}), say, $H^1_0(\Omega)$ or
$H^1_{\#}(\Omega)$.

Assumptions (\ref{eq-approx-simple}) is apparently very mild and satisfied by several typical finite dimensional subspaces used in practice, for
instance, spaces spanned by plane wave bases \cite{canuto07},
  spaces spanned by wavelets
\cite{arias99,genovese08},
 and piecewise polynomial finite element spaces
\cite{ciarlet78}.
As a result, we may investigate all these kinds of finite
dimensional approximation approaches in computational either physics
or quantum chemistry in a unified framework. For convenience, here
and hereafter we consider the case of $X=H^1_0(\Omega)$ only.

We see that finite dimensional subspaces
\begin{eqnarray*}
\mathcal{H}_n\equiv S_n^N\subset\mathcal{H}
\end{eqnarray*}
satisfying
\begin{eqnarray}\label{eq-approx}
\lim_{n\to\infty}\inf_{\Psi\in\mathcal{H}_n}
\|\Psi-\Phi\|_{1,\Omega}=0 \quad \forall~\Phi\in\mathcal{H}.
\end{eqnarray}
We shall study the numerical analysis of the following minimization
problem
\begin{eqnarray}\label{problem-min-dis}
\inf\{ E(\Phi_n):\Phi_n\in\mathcal{H}_n\cap\mathbb{Q} \}.
\end{eqnarray}
The existence of a minimizer of \eqref{problem-min-dis} can be
obtained by similar arguments to that in the proof of Theorem
\ref{theo-convergence} (c.f., also, \cite{cancesM10,chen-gong09}).
However, the uniqueness is unknown even up to a unitary transform.
Therefore we define the set of finite dimensional ground state
solutions:
\begin{eqnarray*}\label{set-ground-dis}
\mathcal{G}_n=\left\{\Phi_n\in\mathcal{H}_n\cap\mathbb{Q}:
E(\Phi_n)=\min_{\Psi\in\mathcal{H}_n\cap\mathbb{Q}}E(\Psi)\right\}.
\end{eqnarray*}
Given $ n\ge 1$, any minimizer
$\Phi_n=(\phi_{1,n},\phi_{2,n},\cdots,\phi_{N,n})$ of
\eqref{problem-min-dis} solves
\begin{eqnarray}\label{problem-eigen-dis}
\left\{ \begin{array}{rcl} (A_{\Phi_n}\phi_{i,n},v) &=&
\displaystyle \big( \sum_{j=1}^N \lambda_{ij,n}\phi_{j,n},v \big)
\quad\forall ~v\in S_n, \quad i=1,2,\cdots,N,\\[1ex]
\displaystyle \int_{\Omega}\phi_{i,n}\phi_{j,n} &=& \delta_{ij}
\end{array} \right.
\end{eqnarray}
with the Lagrange multiplier
\begin{eqnarray}\label{eq-Lambda-dis}
\Lambda_n=(\lambda_{ij,n})_{i,j=1}^N =\left(
\int_{\Omega}\phi_{j,n}A_{\Phi_n}\phi_{i,n} \right)_{i,j=1}^N.
\end{eqnarray}
Define the set of finite dimensional ground state eigenpairs
\begin{eqnarray*}\label{set-ground-dis}
\Theta_n=\left\{(\Lambda_n,\Phi_n)\in\mathbb{R}^{N\times N}\times
(\mathcal{H}_n\cap\mathbb{Q}): \Phi_n\in \mathcal{G}_n \mbox{ and
}(\Lambda_n,\Phi_n)\mbox{ solves }\eqref{problem-eigen-dis}
\right\}.
\end{eqnarray*}
Proposition \ref{proposition-coercive} and \eqref{eq-Lambda-dis} imply that the finite dimensional approximations are uniformly bounded
\begin{eqnarray}\label{eq-bounded}
\sup_{(\Lambda_n,\Phi_n)\in \Theta_n, n\geq 1}(\|\Phi_n\|_{1,\Omega}+|\Lambda_n|)<C
\end{eqnarray}
for some constant $C$.

 We then address the Galerkin discretization of \eqref{problem-abstract}. Let
\begin{gather*}
Y_n=\mathbb{R}^{N\times N}\times \mathcal{H}_n
\end{gather*}
and $F_n:Y_n\rightarrow Y_n^*$ be an approximation of $F$ defined by
\begin{eqnarray}\label{eq-F-h}\nonumber
\langle F_n((\Lambda_n,\Phi_n)),(\bbchi_n,\Gamma_n) \rangle =
\langle F((\Lambda_n,\Phi_n)),(\bbchi_n,\Gamma_n) \rangle
\quad\forall~(\Lambda_n,\Phi_n), (\bbchi_n,\Gamma_n)\in Y_n.
\end{eqnarray}
Then  discrete problem \eqref{problem-eigen-dis} can  be rewritten as
\begin{eqnarray}\label{problem-abstract-dis}
F_n((\Lambda_n,\Phi_n))=0\in Y_n^*.
\end{eqnarray}
We also denote the derivative of $F_n$ at $(\Lambda_n,\Phi_n)\in
Y_n$ by $F'_{n,(\Lambda_n,\Phi_n)}:Y_n\rightarrow Y_n^*$ as follows:
\begin{eqnarray*}\label{eq-F-derivative-h}\nonumber
\langle F'_{n,(\Lambda_n,\Phi_n)}((\bbmu_n,\Psi_n)),(\bbchi_n,\Gamma_n) \rangle &=& \langle \mathcal{L}'_{\Phi_n}(\Lambda_n,\Phi_n)\Psi_n,\Gamma_{n}
\rangle -\sum_{i,j=1}^N (\mu_{ij,n}\phi_{j,n}, \gamma_{i,n}) \\&& + \sum_{i,j=1}^N \chi_{ij,n}\int_{\Omega}(\psi_{i,n}\phi_{j,n}+\phi_{i,n}\psi_{j,n}).
\end{eqnarray*}
Given $(\Lambda,\Phi)\in \mathcal{S}^{N\times N}\times \mathbb{Q}$,
we define
$$
X_{\Phi}=\mathcal{S}^{N\times N}\times
(\mathcal{S}_{\Phi}\oplus\mathcal{T}_{\Phi})\subset Y
$$
with the induced norm $\|(\bbmu,\Psi)\|_{X_{\Phi}}=|\bbmu|+\|\Psi\|_{1,\Omega}$
for each $(\bbmu,\Psi)\in X_{\Phi}$
and
\begin{gather*}
X_{\Phi,n}=\mathcal{S}^{N\times N}\times(\mathcal{H}_n\cap
(\mathcal{S}_{\Phi}\oplus\mathcal{T}_{\Phi})).
\end{gather*}

We assume here and hereafter that $y_0\equiv(\Lambda_0,\Phi_0)$ is a solution of \eqref{problem-eigen-compact-L} satisfying \eqref{assumption-a2}, where
$\Lambda_0=(\lambda_{0,ij})^N_{i,j=1}$ and $\Phi_0=(\phi_{0,1},\phi_{0,2},\cdots,\phi_{0,N})$.
We shall derive the existence of a unique local discrete solution
$y_n\in X_{\Phi_0,n}$ of \eqref{problem-eigen-dis}
 in the neighborhood
of $y_0$.

\begin{lemma}\label{lemma-isomorphism}
 $F'_{y_0}:X_{\Phi_0}\rightarrow X_{\Phi_0}^*$ is an isomorphism.
\end{lemma}

\begin{proof}
It is sufficient to prove that  equation
\begin{eqnarray}\label{eq-solvable}
F'_{y_0} ((\bbmu,\Psi)) = (\bbeta,g)
\end{eqnarray}
is uniquely solvable in $X_{\Phi_0}$ for every $(\bbeta,g)\in
X_{\Phi_0}^*$. To this end we define the following bilinear forms
$a_{\Phi_0}:\mathcal{H}\times\mathcal{H}\rightarrow\mathbb{R}$ and
$b_{\Phi_0},c_{\Phi_0}:\mathcal{H}\times\mathbb{R}^{N\times
N}\rightarrow\mathbb{R}$ by
\begin{eqnarray*}
a_{\Phi_0}(\Psi,\Gamma) &=& \langle
\mathcal{L}'_{\Phi_0}(\Lambda_0,\Phi_0)\Psi,\Gamma \rangle, \\[1ex]
b_{\Phi_0}(\Psi,\bbchi) &=& \sum_{i,j=1}^N \chi_{ij}(\phi_{0,i},\psi_j), \\[1ex]
c_{\Phi_0}(\Psi,\bbchi) &=& \sum_{i,j=1}^N \chi_{ij} \big(
(\phi_{0,i},\psi_j)+(\phi_{0,j},\psi_i) \big).
\end{eqnarray*}

Using \eqref{eq-F-derivative}, we may rewrite \eqref{eq-solvable} as
follows: find $\bbmu\in\mathcal{S}^{N\times N}$ and
$\Psi\in\mathcal{S}_{\Phi_0}\oplus\mathcal{T}_{\Phi_0}$ such that
\begin{eqnarray}\label{problem-saddle}
\left\{ \begin{array}{rcll} a_{\Phi_0}(\Psi,\Gamma) -
b_{\Phi_0}(\Gamma,\bbmu) &=& (g,\Gamma) &\quad\forall
~\Gamma\in\mathcal{S}_{\Phi_0}\oplus\mathcal{T}_{\Phi_0}, \\[1ex]
c_{\Phi_0}(\Psi,\bbchi) &=& \displaystyle\sum_{i,j=1}^N
\chi_{ij}\eta_{ij} &\quad\forall ~\bbchi\in\mathcal{S}^{N\times N}.
\end{array} \right.
\end{eqnarray}

For any given $\bbchi\in\mathcal{S}^{N\times N}$, we can choose
$\Psi=\Phi_0\bbchi$, and thus
\begin{eqnarray}\label{proof-14}
c_{\Phi_0}(\Psi,\bbchi)=2\sum_{i,j=1}^N |\chi_{ij}|^2,
\end{eqnarray}
where $\Phi_0^T\Phi_0=I^{N\times N}$ is used.
Note that a simple calculation leads to
\begin{eqnarray}\label{proof-15}
\|\Psi\|_{1,\Omega}=\|\Phi_0\bbchi\|_{1,\Omega} \lesssim
(\sum_{i,j=1}^N |\chi_{ij}|^2)^{1/2}\|\Phi_0\|_{1,\Omega}.
\end{eqnarray}
By taking into account \eqref{eq-coercive}, \eqref{proof-14} and
\eqref{proof-15}, we obtain
\begin{eqnarray}\label{proof-12}
\inf_{\chi\in\mathcal{S}^{N\times N}}
\sup_{\Psi\in\mathcal{S}_{\Phi_0}}
\frac{c_{\Phi_0}(\Psi,\bbchi)}{\|\Psi\|_{1,\Omega} (\sum_{i,j=1}^N
|\chi_{ij}|^2)^{1/2}} ~\geq~ \kappa_c,
\end{eqnarray}
where $\kappa_c>0$ is independent of $\bbchi$.
Hence, there exists a unique solution $\Psi_S\in
\mathcal{S}_{\Phi_0}$ such that
$$
c_{\Phi_0}(\Psi_S,\bbchi) = \displaystyle\sum_{i,j=1}^N
\chi_{ij}\eta_{ij} \quad\forall ~\bbchi\in\mathcal{S}^{N\times N}.
$$
Therefore \eqref{problem-saddle} is equivalent to: find
$\Psi_0\in\mathcal{T}_{\Phi_0}$ such that
\begin{eqnarray}\label{proof-a1}
a_{\Phi_0}(\Psi_0,\Gamma) = (g,\Gamma)-a_{\Phi_0}(\Psi_S,\Gamma)
\quad\forall ~\Gamma\in\mathcal{T}_{\Phi_0}.
\end{eqnarray}
The unique solvability of (\ref{proof-a1}) is a direct consequence
of \eqref{assumption-a2}.

Using similar arguments to that from \eqref{proof-14} to
\eqref{proof-12}, we get
\begin{eqnarray*}\label{proof-a2}
\inf_{\chi\in\mathcal{S}^{N\times N}}
\sup_{\Psi\in\mathcal{S}_{\Phi_0}}
\frac{b_{\Phi_0}(\Psi,\bbchi)}{\|\Psi\|_{1,\Omega} (\sum_{i,j=1}^N
|\chi_{ij}|^2)^{1/2}} ~\geq~ \kappa_b,
\end{eqnarray*}
where $\kappa_b>0$ is independent of $\bbchi$. This implies that
equation
\begin{eqnarray*}
b_{\Phi_0}(\Gamma,\bbmu) =
 a_{\Phi_0}(\Psi_0+\Psi_S,\Gamma)-(g,\Gamma) \quad\forall
~\Gamma\in\mathcal{S}_{\Phi_0}  
\end{eqnarray*}
has a unique solution $\bbmu_S\in \mathcal{S}^{N\times N}$.

We have proved that for any $(\bbeta,g)\in X_{\Phi_0}^*$ in
\eqref{problem-saddle}, there exists a unique solution
$(\bbmu_S,\Psi_0+\Psi_S)$. This indicates that $F'_{y_0}$ is an
isomorphism from $X_{\Phi_0}$ to $X_{\Phi_0}^*$ and completes the
proof.
\end{proof}

Note that $F'_{y_0}:X_{\Phi_0}\rightarrow X_{\Phi_0}^*$ being an
isomorphism is equivalent to the following inf-sup condition
\begin{eqnarray}\label{inf-sup}
\inf_{y_1\in X_{\Phi_0}}\sup_{y_2\in X_{\Phi_0}}\frac{\langle
F'_{y_0} y_1,y_2 \rangle}{\|y_1\|_{X_{\Phi_0}}\|y_2\|_{X_{\Phi_0}}}
= \beta>0
\end{eqnarray}
with the constant satisfying $\beta^{-1}=\|F'_{y_0}\hskip -0.2cm
^{-1}\|$.

For any $\Phi\in\mathbb{Q}$, we define
\begin{eqnarray*}
\mathbb{Q}^{\Phi}=\{\Psi\in\mathbb{Q}:
\|\Psi-\Phi\|_{0,\Omega}=\min_{U\in \mathcal{O}^{N\times N}}\|\Psi
U-\Phi\|_{0,\Omega}\}.
\end{eqnarray*}
In our analysis, we need the following lemma, whose proof is
referred to \cite{cancesM10}.
\begin{lemma}\label{lemma-decomposition}
If $\Phi\in\mathbb{Q}$, then $\Psi\in \mathbb{Q}^{\Phi}$ can be
represented by
\begin{eqnarray*}\label{eq-decomposition}
\Psi=\Phi+S(W)\Phi+W, 
\end{eqnarray*}
where $W\in\mathcal{T}_{\Phi}$ and $S(W)\in \mathcal{S}^{N\times N}$
satisfying
\begin{eqnarray}\label{property-SW}
|\mathcal{S}(W)|=|(I^{N\times N}-W^TW)^{1/2}-I^{N\times N}|\leq
\|W\|_{0,\Omega}^2\leq \|\Psi-\Phi\|_{0,\Omega}^2.
\end{eqnarray}
\end{lemma}

Before giving a discrete counterpart with Lemma
\ref{lemma-isomorphism}, we also  need to introduce two projections.
First, we define the projection $\tilde{\Pi}_n:\mathbb{Q}\rightarrow
\mathcal{H}_n\cap\mathbb{Q}$ such that
\begin{eqnarray*}\label{proof-a4}
\|\tilde{\Pi}_n\Phi-\Phi\|_{1,\Omega}=\min_{\Psi\in
\mathcal{H}_n\cap\mathbb{Q}} \|\Psi-\Phi\|_{1,\Omega}\quad \forall~
\Phi\in\mathbb{Q}.
\end{eqnarray*}
To project further into $X_{\Phi,n}$, we then define 
$\Pi_n:\mathcal{S}^{N\times N}\times 
\mathbb{Q}\rightarrow X_{\Phi,n}$ by
\begin{eqnarray*}
\Pi_n (\Lambda,\Phi)=(\Lambda,(\tilde{\Pi}_n\Phi) \tilde{U})\qquad
\forall~(\Lambda,\Phi)\in \mathcal{S}^{N\times N}\times \mathbb{Q},
\end{eqnarray*}
where
\begin{eqnarray*}
\tilde{U}= \arg \min_{U\in \mathcal{O}^{N\times
N}}\|(\tilde{\Pi}_n\Phi)U -\Phi\|_{0,\Omega}.
\end{eqnarray*}
From Lemma \ref{lemma-decomposition}, we see that $\Pi_n:\mathcal{S}^{N\times N}\times 
\mathbb{Q}\rightarrow X_{\Phi,n}$ is well-defined.
\begin{lemma}\label{lemma-h-isomorphism}
If Assumption {\bf A2} is satisfied, then there exists $n_0>1$ such
that $F'_{n,\Pi_n y_0}: X_{\Phi_0,n}\rightarrow X_{\Phi_0,n}^*$ is
an isomorphism for all $n\geq n_0$.
Moreover, there is a constant $M>0$ such that
\begin{eqnarray*}\label{eq-bound}
\|F'_{n,\Pi_n y_0}\hskip -0.2cm ^{-1}\|\leq M \quad\forall~n\geq
n_0.
\end{eqnarray*}
\end{lemma}

\begin{proof}
We first prove that
\begin{eqnarray}\label{pi-n-convergence}
\lim_{n\rightarrow\infty}\|\Pi_n y-y\|_{X_{\Phi}}=0 \qquad
\forall~y\equiv(\Lambda,\Phi)\in \mathcal{S}^{N\times N}\times
\mathbb{Q}.
\end{eqnarray}
Using the fact that $\Phi\in \mathbb{Q}$ and
$(\tilde{\Pi}_n\Phi)\tilde{U}\in \mathbb{Q}^{\Phi}$, we have
\begin{eqnarray*}
|\tilde{U}-I| =
\|(\tilde{\Pi}_n\Phi)\tilde{U}-\tilde{\Pi}_n\Phi\|_{0,\Omega} \leq
\|(\tilde{\Pi}_n\Phi)\tilde{U}-\Phi\|_{0,\Omega}
+\|\tilde{\Pi}_n\Phi-\Phi\|_{0,\Omega} \lesssim
\|\tilde{\Pi}_n\Phi-\Phi\|_{1,\Omega},
\end{eqnarray*}
which implies
\begin{eqnarray}\label{proof-ad} \nonumber
\|(\tilde{\Pi}_n\Phi)\tilde{U}-\Phi\|_{1,\Omega} &\leq&
\|\tilde{\Pi}_n\Phi-\Phi\|_{1,\Omega}
+\|(\tilde{\Pi}_n\Phi)\tilde{U}-\tilde{\Pi}_n\Phi\|_{1,\Omega} \\
\nonumber &\leq& \|\tilde{\Pi}_n\Phi-\Phi\|_{1,\Omega}+|\tilde{U}-I|
\cdot\|\tilde{\Pi}_n\Phi\|_{1,\Omega} \\
&\lesssim& \|\tilde{\Pi}_n\Phi-\Phi\|_{1,\Omega}.
\end{eqnarray}
Let
$\Phi^n\equiv(\phi_1^n,\phi_2^n,\cdots,\phi^n_N)=\displaystyle\arg\min_{\Psi\in\mathcal{H}_n}
\|\Psi-\Phi\|_{1,\Omega}$, we may estimate
$\|\tilde{\Pi}_n\Phi-\Phi\|_{1,\Omega}$ as follows:
\begin{eqnarray*}
\|\tilde{\Pi}_n\Phi-\Phi\|_{1,\Omega} &\leq& \sum_{i=1}^N
\|\frac{Q_n\phi^n_i}{\|Q_n\phi^n_i\|_{0,\Omega}}-\phi_i\|_{1,\Omega} \\
&\leq& \sum_{i=1}^N (\|Q_n\phi^n_i-\phi_i\|_{1,\Omega}
+\|\frac{Q_n\phi^n_i}{\|Q_n\phi^n_i\|_{0,\Omega}}-Q_n\phi^n_i\|_{1,\Omega}) \\
&\leq&
\sum_{i=1}^N(1+\frac{\|Q_n\phi^n_i\|_{1,\Omega}}{\|Q_n\phi^n_i\|_{0,\Omega}})
\|\phi_i-Q_n\phi^n_i\|_{1,\Omega},
\end{eqnarray*}
where $Q_n$ is the Gram-Schmidt orthogonal operator:
$$
Q_n\phi^n_i=\phi^n_i-\displaystyle\sum_{j=1}^{i-1}
\frac{(Q_n\phi^n_j,\phi^n_i)}{(Q_n\phi^n_j,Q_n\phi^n_j)}Q_n\phi_j^n
\quad\quad i=1,\cdots,N.
$$
Note that
\begin{eqnarray*}
\|\phi_i-Q_n\phi^n_i\|_{1,\Omega}
&\leq&\|\phi_i^n-\phi_i\|_{1,\Omega} +\sum_{j=1}^{i-1}
\frac{\|Q_n\phi_j^n\|_{1,\Omega}}{\|Q_n\phi_j^n\|_{0,\Omega}^2}
((Q_n\phi^n_j,\phi^n_i-\phi_i) + (Q_n\phi^n_j-\phi_j,\phi_i)) \\
&\leq& (1+\sum_{j=1}^{i-1}
\frac{\|Q_n\phi_j^n\|_{1,\Omega}}{\|Q_n\phi_j^n\|_{0,\Omega}})
\|\phi_i-\phi_i^n\|_{1,\Omega} + \sum_{j=1}^{i-1}
\frac{\|Q_n\phi_j^n\|_{1,\Omega}}{\|Q_n\phi_j^n\|_{0,\Omega}^2}
\|\phi_j-Q_n\phi^n_j\|_{1,\Omega},
\end{eqnarray*}
we conclude
\begin{eqnarray}\label{proof-a3}
\|\tilde{\Pi}_n\Phi-\Phi\|_{1,\Omega}\lesssim
\|\Phi^n-\Phi\|_{1,\Omega}
=\inf_{\Psi\in\mathcal{H}_n}\|\Psi-\Phi\|_{1,\Omega}.
\end{eqnarray}
Using \eqref{proof-ad}, \eqref{proof-a3} and the definition of
$\Pi_n$, we arrive at
\begin{eqnarray}\label{pi-n-convergence--}
\|\Pi_n y-y\|_{X_{\Phi}} \lesssim
\inf_{\Psi\in\mathcal{H}_n}\|\Psi-\Phi\|_{1,\Omega},
\end{eqnarray}
which together with \eqref{eq-approx} leads to
\eqref{pi-n-convergence}.

We then show the invertibility of
$F'_{n,y_0}:X_{\Phi_0,n}\rightarrow X_{\Phi_0,n}^*$. We obtain from
\eqref{inf-sup} that
\begin{eqnarray*}
\sup_{y_{2}\in X_{\Phi_0}}\frac{\langle F'_{y_0} y_{1},y_{2}
\rangle}{\|y_{1}\|_{X_{\Phi_0}}\|y_{2}\|_{X_{\Phi_0}}} \geq \beta
\quad\forall~y_{1}\in X_{\Phi_0,n}.
\end{eqnarray*}

Let $P_n^{\Phi_0}: \mathcal{S}_{\Phi_0}\cap \mathcal{T}_{\Phi_0}\to
\mathcal{H}_n\cap(\mathcal{S}_{\Phi_0}\cap \mathcal{T}_{\Phi_0})$ be
a projection operator satisfying
\begin{eqnarray*}
\big( \nabla \Phi_1, \nabla (\Phi_2-P_n^{\Phi_0}
\Phi_2)\big)=0\qquad
\forall~\Phi_1\in\mathcal{H}_n\cap(\mathcal{S}_{\Phi_0}\cap
\mathcal{T}_{\Phi_0}).
\end{eqnarray*}
Set
$$\eta_n=\sup_{\Psi\in \mathcal{S}_{\Phi_0}\cap \mathcal{T}_{\Phi_0},\|\Psi\|_{1,\Omega}\le 1}\|\Psi-P_n^{\Phi_0} \Psi\|_{0,\Omega},$$
we have (see, e.g., \cite{zhou10})
\begin{eqnarray}\label{L2-H1-projection}
\|\Psi-P_n^{\Phi_0} \Psi\|_{0,\Omega}\lesssim \eta_n
\|\Psi\|_{1,\Omega} \quad\forall~\Psi\in \mathcal{S}_{\Phi_0}\cap
\mathcal{T}_{\Phi_0}\quad \textnormal{with} \quad \lim_{n\to
\infty}\eta_n=0.
\end{eqnarray}
Let $P_n=(I,P_n^{\Phi_0})$, we obtain from definition
\eqref{eq-F-derivative} and \eqref{L2-H1-projection} that
\begin{eqnarray*}
\langle F'_{y_0} y_1, P_n y_2\rangle &=&\langle F'_{y_0} y_1,
y_2\rangle -
\langle F'_{y_0} y_1,y_2-P_n y_2\rangle\nonumber\\
&=& \langle F'_{y_0} y_1, y_2\rangle + \frac{1}{2} ( \nabla \Phi_1,
\nabla (\Phi_2-P_n^{\Phi_0} \Phi_2)) - \langle F'_{y_0} y_1, y_2-P_n
y_2\rangle\nonumber\\
&\geq & \langle F'_{y_0} y_1, y_2\rangle - c
\|y_1\|_{X_{\Phi_0}}\|y_2-P_ny_2\|_{0,\Omega}\nonumber\\
&\geq & \langle F'_{y_0} y_1, y_2\rangle - c\eta_n
\|y_1\|_{X_{\Phi_0}}\|y_2\|_{X_{\Phi_0}},
\end{eqnarray*}
which  implies that there exists $\tilde{n}$ such that for all
$n\ge\tilde{n}$, there holds
\begin{eqnarray*}
\sup_{y_{2}\in X_{\Phi_0,n}}\frac{\langle F'_{y_0} y_{1},y_{2}
\rangle}{\|y_{1}\|_{X_{\Phi_0}}\|y_{2}\|_{X_{\Phi_0}}} \geq
\frac{\beta}{2} \quad\forall~y_{1}\in X_{\Phi_0,n},
\end{eqnarray*}
or equivalently
\begin{eqnarray*}
\inf_{y_{1}\in X_{\Phi_0,n}}\sup_{y_{2}\in
X_{\Phi_0,n}}\frac{\langle F'_{y_0} y_{1},y_{2}
\rangle}{\|y_{1}\|_{X_{\Phi_0}}\|y_{2}\|_{X_{\Phi_0}}} \geq
\frac{\beta}{2}.
\end{eqnarray*}
Thus $F'_{n,y_0}$ is an isomorphism from $X_{\Phi_0,n}$ to
$X_{\Phi_0,n}^*$ satisfying
\begin{eqnarray*}\label{eq-bound-2}
\|F'_{n,y_0}\hskip -0.2cm ^{-1}\|\leq 2\beta^{-1} \quad\forall~n\geq
\tilde{n}.
\end{eqnarray*}
Note that $F'_n$ satisfies the following discrete H\"{o}lder
condition
\begin{eqnarray*}\label{dis-F-conti}
\|F'_{n,y_0}-F'_{n,\Pi_n y_0}\| \lesssim \|y_0-\Pi_n
y_0\|^{\alpha}_{X_{\Phi_0}}+\|y_0-\Pi_n y_0\|^{2}_{X_{\Phi_0}}.
\end{eqnarray*}
It follows from \eqref{pi-n-convergence} that there exists
$n_0>\tilde{n}$ such that the inf-sup constant of $F'_{n,\Pi_n y_0}$
is uniformly away from zero for all $n\geq n_0$.
This completes the proof.
\end{proof}

\begin{theorem}\label{unique}
 If Assumption {\bf  A2} is satisfied, then
there exist  $\delta>0$, $n_1> n_0$ such that
\eqref{problem-eigen-dis} has a unique local solution
$y_n=(\Lambda_n,\Phi_n)\in X_{\Phi_0,n}\cap B_{\delta}(y_0)$ for all
$n\geq n_1$.
\end{theorem}
\begin{proof}
The idea is to construct a
contractive mapping whose fixed point is  $y_n$. We rewrite
\eqref{problem-abstract-dis} as
$$
F_n(y_n) - F_n(\Pi_n y_0) = -F_n(\Pi_n y_0).
$$
Using \eqref{eq-lipschitz}, we have
\begin{eqnarray*}\label{proof-11}\nonumber
\|F_n(\Pi_n y_0)\|_{X_{\Phi_0,n}^*} &=& \|F(\Pi_n y_0)|_{X_{\Phi_0,n}}-F(y_0)|_{X_{\Phi_0,n}}\|_{X_{\Phi_0,n}^*}\\
&\leq &\|F(\Pi_n y_0)-F(y_0)\|_{X_{\Phi_0}^*}~\lesssim~\|y_0-\Pi_n
y_0\|_{X_{\Phi_0}}.
\end{eqnarray*}
From Lemma \ref{lemma-h-isomorphism}, we may  define the map
$\mathcal{N} :B_R(\Pi_n y_0)\cap X_{\Phi_0,n} \rightarrow
X_{\Phi_0,n}$ by
\begin{eqnarray*}
F'_{n,\Pi_n y_0}(\mathcal{N}(x)-\Pi_n y_0)=-F_n(\Pi_n y_0) -
(x-\Pi_n y_0) \int_0^1\big(F'_{n,\Pi_n y_0+t(x-\Pi_n
y_0)}-F'_{n,\Pi_n y_0}\big)dt
\end{eqnarray*}
when $n\geq n_0$.

We will show that $\mathcal{N}$ is a contraction from $B_R(\Pi_n
y_0)\cap X_{\Phi_0,n}$ into $B_R(\Pi_n y_0)\cap X_{\Phi_0,n}$ if $R$
is chosen sufficiently small and $n$ is large enough.

First, we prove that $\mathcal{N}$ maps $B_R(\Pi_n y_0)\cap
X_{\Phi_0,n}$ to $B_R(\Pi_n y_0)\cap X_{\Phi_0,n}$ for sufficiently
small $R$. Note that $F'_{n,\Pi_n y_0}$ is an isomorphism on
$X_{\Phi_0,n}$ if $n$ is sufficiently large. For each $x\in
B_R(\Pi_n y_0)$, we have $\mathcal{N}(x)-\Pi_n y_0 \in X_{\Phi_0,n}$
and
\begin{eqnarray*}
&& \|\mathcal{N}(x)-\Pi_n y_0\|_{X_{\Phi_0}} \\\nonumber
&\leq&M\big(\|F_n(\Pi_n
y_0)\|_{X_{\Phi_0,n}^*} + R\int_0^1 \|F'_{n,\Pi_n y_0+t(x-\Pi_n y_0)}-F'_{n,\Pi_n y_0}\|dt\big)\\
&\leq & CM\big(\|\Pi_n
y_0-y_0\|_{X_{\Phi_0}}+R(R^{\alpha}+R^2)\big).
\end{eqnarray*}
Since $CM(\|\Pi_n y_0-y_0\|_{X_{\Phi_0}}+R^{1+\alpha}+R^3)$ can be
estimated by $R$ when $R$ is sufficiently small and $n$ is
sufficiently large, we have that $\mathcal{N}(x) \in B_R(\Pi_n
y_0)$. It is clear that $R$ can be chosen independently of $n$.

Next, we show that $\mathcal{N}$ is a contraction on $B_R(\Pi_n
y_0)\cap X_{\Phi_0,n}$. If $x_1, x_2 \in B_R(\Pi_n y_0)\cap
X_{\Phi_0,n}$, then
\begin{eqnarray*}
F'_{n,\Pi_n y_0}(\mathcal{N}(x_1)-\mathcal{N}(x_2)) =
(x_1-x_2)\int_0^1\big(F'_{n,\Pi_n y_0}-F'_{n,x_1+t(x_2-x_1)}\big)dt.
\end{eqnarray*}
Thus, $\|\mathcal{N}(x_1)-\mathcal{N}(x_2)\|_{X_{\Phi_0}}$ can be
estimated as
\begin{eqnarray*}
&& \|\mathcal{N}(x_1)-\mathcal{N}(x_2)\|_{X_{\Phi_0}} \\&\leq &
M\|x_2-x_1\|_{X_{\Phi_0}} \int_0^1\big \|F'_{n,\Pi_n y_0}-F'_{n, x_1+t(x_2-x_1)}\big\|dt\\
&\leq & CM(R^{\alpha}+R^2)\|x_1-x_2\|_{X_{\Phi_0}}.
\end{eqnarray*}
 We obtain for
sufficiently small $R$ that $CM(R^{\alpha}+R^2)<1$  and hence
$\mathcal{N}$ is a contraction on $B_R(\Pi_n y_0)$.

We are now able to use Banach's Fixed Point Theorem to obtain the
existence and uniqueness of a fixed point $y_n$ of  map $\mathcal{N}
: B_R(\Pi_n y_0)\cap X_{\Phi_0,n} \rightarrow B_R(\Pi_n y_0)\cap
X_{\Phi_0,n}$, which is the solution of $F_n(y_n)=0$. This completes
the proof.
\end{proof}

\section{Numerical analysis}\label{sec-error-estimate}\setcounter{equation}{0}
In this section, we shall prove the convergence of finite
dimensional approximations and derive various error estimates under
different assumptions.

\subsection{Convergence}
\label{subsec-convergence} The purpose of this subsection is to
prove the convergence of the numerical ground state solutions, for
which we need to introduce the following distances between two sets.
We define the distance
 between two subsets $A,B\subset Y$
by
\begin{eqnarray*}\label{def-dis-set}
\mathcal{D}(A, B)=\sup_{(\Lambda,\Phi)\in A}\inf_{(\bbmu,\Psi)\in B}
(|\Lambda-\bbmu|+\|\Phi-\Psi\|_{1,\Omega})
\end{eqnarray*}
and the  distance between two sets $M,N\subset\mathcal{H}$ by
\begin{eqnarray*}\label{dis-H}
d_{\mathcal{H}}(M,N)=\sup_{\Phi\in M}\inf_{\Psi\in N}
\|\Phi-\Psi\|_{1,\Omega}.
\end{eqnarray*}
\begin{theorem}\label{theo-convergence}
There hold
\begin{eqnarray}\label{conv-phi}
\lim_{n\to\infty} \mathcal{D}(\Theta_{n}, \Theta)=0,
\end{eqnarray} \vskip -0.5cm
\begin{eqnarray}\label{conv-energy}
\lim_{n\to\infty}E_n=\min_{\Psi\in\mathbb{Q}}E(\Psi),
\end{eqnarray}
where $E_n=E(\Phi_n)$ for any $\Phi_n\in\mathcal {G}_n$.
\end{theorem}

\begin{proof}
Let $(\Lambda_n,\Phi_{n})\in\Theta_n$ for $n=1,2,\cdots$. Given any
subsequence $\{\Phi_{n_{k}}\}$ of $\{\Phi_{n}\}$ with $1\leq
n_1<n_2<\cdots<n_k<\cdots$, we obtain from the Banach-Alaoglu
Theorem and \eqref{eq-bounded} that there exist $\Phi\in
\mathcal{H}$ and a weakly convergent subsequence $\{\Phi_{n_{k_j}}\}
\subset\{\Phi_{n_k}\}$ such that
\begin{eqnarray}\label{proof-th1-1}
\Phi_{n_{k_j}}\rightharpoonup \Phi ~~\mbox{in}~\mathcal{H}.
\end{eqnarray}
Next we shall prove $\Phi\in \mathcal{G}$ and
\begin{eqnarray}\label{proof-th1-2}
\lim_{j\to\infty}\|\Phi-\Phi_{n_{k_j}}\|_{1,\Omega}=0,
\end{eqnarray}\vskip -0.5cm
\begin{eqnarray}\label{proof-th1-3}
\lim_{j\to\infty}E(\Phi_{n_{k_j}})=\min_{\Psi\in\mathbb{Q}}E(\Psi).
\end{eqnarray}

From (\ref{proof-th1-1}) and Proposition
\ref{proposition-lower-semi-continuous}, we have
\begin{eqnarray}\label{proof-th1-4}
\liminf_{j\to\infty} E(\Phi_{n_{k_j}})\geq E(\Phi).
\end{eqnarray}
Note that \eqref{eq-approx} implies that $\{\Phi_{n_{k_j}}\}$ is a
minimizing sequence for $E(\Psi)$ and the Rellich theorem shows
that $$\int_{\Omega}\phi_{i,{n_{k_j}}}\phi_{j,{n_{k_j}}}\rightarrow
\int_{\Omega}\phi_{i}\phi_{j}\quad j\to \infty.$$  Therefore
$\Phi\in \mathbb{Q}$ is a minimizer of $E(\Psi)$, which together
with (\ref{proof-th1-4}) leads to
\begin{eqnarray}\label{proof-th1-5}
\lim_{j\to\infty}E(\Phi_{n_{k_j}})=E(\Phi)
=\min_{\Psi\in\mathbb{Q}}E(\Psi).
\end{eqnarray}
This further implies (\ref{proof-th1-3}) and $\Phi\in \mathcal{G}$.

Since $H_0^1(\Omega)$ is compactly imbedded into $L^p(\Omega)$ for
$p\in[2,6)$, we have that $\phi_{i,n_{k_j}}\to \phi_i$ strongly in
$L^p(\Omega)$ as $j\rightarrow \infty$ for $i=1,2,\cdots,N$. This
indicates that $\{\rho_{\Phi_{n_{k_j}}}\}$ converges to
$\rho_{\Phi}$ strongly in $L^q(\Omega)$ for $q\in[1,3)$, from which
we obtain that
\begin{eqnarray*}\label{proof-5}
\lim_{j\to\infty}\int_{\Omega}V_{loc}(x)(\rho_{\Phi_{n_{k_j}}}(x)
-\rho_{\Phi}(x))dx=0,
\end{eqnarray*}\vskip -0.4cm
\begin{eqnarray*}\label{proof-6}
\lim_{j\to\infty}\int_{\Omega}\big(\mathcal{E}
(\rho_{\Phi_{n_{k_j}}})-\mathcal{E}(\rho_{\Phi}(x))\big)dx=0,
\end{eqnarray*}
and
\begin{eqnarray}\label{proof-7}
\lim_{j\to\infty}D(\rho_{\Phi_{n_{k_j}}},\rho_{\Phi_{n_{k_j}}}) =
D(\rho_{\Phi},\rho_{\Phi}).
\end{eqnarray}

Consequently, we can get from \eqref{proof-th1-5}  to
\eqref{proof-7} that each term of $E(\cdot)$ converges and in
particular
\begin{eqnarray*}\label{con1}
\lim_{j\to\infty}\sum_{i=1}^N\|\nabla\phi_{i,n_{k_j}}\|^2_{0,\Omega}=
\sum_{i=1}^N\|\nabla\phi_{i}\|^2_{0,\Omega}.
\end{eqnarray*}
Using \eqref{proof-th1-1} and the fact that $\mathcal{H}$ is a
Hilbert space under norm
$\displaystyle\left(\sum_{i=1}^N\|\nabla\phi_i\|^2_{0,\Omega}\right)^{1/2}$,
we obtain (\ref{proof-th1-2}).
If $(\Lambda,\Phi)$ solves \eqref{problem-eigen-compact-L}, then
\begin{eqnarray*}\label{proof-th1-9}
\lim_{j\to\infty}|\Lambda-\Lambda_{n_{k_j}}|=0
\end{eqnarray*}
is a direct consequence of \eqref{eq-Lambda}, \eqref{eq-Lambda-dis}
and \eqref{proof-th1-2}. Hence we arrive at \eqref{conv-phi}. This
completes the proof.
\end{proof}

\begin{remark}
Theorem \ref{theo-convergence} states that all the limit points of
finite dimensional approximations are  ground state solutions. We
note that \cite{sur10}  gave the convergence of ground state energy
approximations only while we provide further convergence  of
approximations of both eigenvalues and eigenfunctions.
\end{remark}

\subsection{Error estimates for the energy approximation}
\label{sec-priori-error-energy} We shall  derive the quadratic
convergence rate of ground state energy approximations, which is a
generalization and improvement of \cite{cancesM10,sur10}.
\begin{theorem}
  \label{thm:error-energy}
   Let $E$ be the ground state
  energy of \eqref{problem-min} and $E_n$ be the ground state energy
  of \eqref{problem-min-dis}, namely, $E=E(\Phi)$ for all
  $\Phi\in\mathcal{G}$ and $E_n=E(\Phi_n)$ for all
  $\Phi_n\in\mathcal{G}_n$.   If Assumption {\bf A1} holds,
  then
  \begin{eqnarray}\label{error-energy}
    |E-E_n|\lesssim d^2_{\mathcal{H}}(\mathcal{G},\mathcal{H}_n).
  \end{eqnarray}
\end{theorem}

\begin{proof}
We see from the definition of ground state energies $E$ and $E_n$
that
$$
0\leq E_n-E\leq E(\Psi)-E
\quad\forall~\Psi\in\mathcal{H}_n\cap\mathbb{Q}.
$$

Following \cite{cancesM10,schneider09}, if Assumption {\bf A1}
holds, we obtain from the Taylor expansion that for any
$\Psi\in\mathbb{Q}$, there holds
\begin{eqnarray}\label{proof-51}
E(\Psi)-E(\Phi)=(E'(\Phi),\Psi-\Phi) + \frac{1}{2}\langle
E''(\xi)(\Psi-\Phi),\Psi-\Phi \rangle,
\end{eqnarray}
where $\xi=\Phi+\delta(\Psi-\Phi)$ with $\delta\in [0,1]$.
Since $\Phi$ is a ground state solution, we get from
\eqref{problem-eigen-compact-L} that
\begin{eqnarray*}
(E'(\Phi),\Psi-\Phi)=2(\Phi\Lambda,\Psi-\Phi)=2(\Phi UU^T\Lambda
U,\Psi U-\Phi U),
\end{eqnarray*}
where the orthogonal transform $U$ diagonalizes the Lagrange
multiplier $\Lambda$ by
$$
U^T\Lambda U={\rm
diag}\{\tilde{\lambda}_1,\cdots,\tilde{\lambda}_N\}.
$$
Denote $\tilde{\Phi}=\Phi U$ and $\tilde{\Psi}=\Psi U$, we have
\begin{eqnarray}\label{proof-52}\nonumber
(E'(\Phi),\Psi-\Phi)&=&2\sum_{i=1}^N\tilde{\lambda}_i
\int_{\Omega}\tilde{\phi}_i(\tilde{\psi}_i-\tilde{\phi}_i)~\lesssim~
\sum_{i=1}^N\|\tilde{\phi}_i-\tilde{\psi}_i\|^2_{0,\Omega}\\
&\lesssim& \|\tilde{\Phi}-\tilde{\Psi}\|^2_{1,\Omega} ~=~
\|\Phi-\Psi\|^2_{1,\Omega}.
\end{eqnarray}
It is observed by a simple calculation that
$$\langle E''(\xi)\Psi,\Gamma \rangle=2\sum_{i=1}^N(A_{\xi}\psi_i,\gamma_i)
+4\sum_{i,j=1}^ND(\xi_i \psi_i,\xi_j \gamma_j)
+4\sum_{i,j=1}^N\int_{\Omega}\mathcal{E}''(\rho_{\xi})\xi_i\psi_i\xi_j\gamma_j$$
and hence
\begin{eqnarray}\label{proof-53}
\langle E''(\xi)(\Psi-\Phi),\Psi-\Phi \rangle\lesssim \|\Psi-\Phi\|^2_{1,\Omega},
\end{eqnarray}
where the hidden constant  depends on the $\mathcal{H}$-norm of $\Psi$.

Taking \eqref{proof-51}, \eqref{proof-52} and \eqref{proof-53} into account,  we have proved that for $\Phi\in\mathcal{G}$ there holds
\begin{eqnarray*}\label{eq-error-energy}
E(\Psi)-E(\Phi)\lesssim\|\Phi-\Psi\|^2_{1,\Omega}\quad\forall~\Psi\in\mathcal{H}_n\cap\mathbb{Q},
\end{eqnarray*}
which together with the definition of $\tilde{\Pi}_n$ and \eqref{proof-a3} implies that $\tilde{\Pi}_n\Phi\in \mathcal{H}_n\cap \mathbb{Q}$ and
\begin{eqnarray*}\label{eq-error-energy}
0\leq E_n-E\leq E(\tilde{\Pi}_n\Phi)-E(\Phi)\lesssim\|\tilde{\Pi}_n\Phi-\Phi\|^2_{1,\Omega}\lesssim d^2_{\mathcal{H}}(\mathcal{G},\mathcal{H}_n),
\end{eqnarray*}
where the hidden constant, by using \eqref{eq-bounded}, is only dependent on the problem. This completes the proof.
\end{proof}

\subsection{Error estimates for ground state solutions}
\label{sec-priori-error} In this subsection, we shall derive the
 a priori error estimates for finite dimensional
approximations of Kohn-Sham equations under Assumptions {\bf A2} and {\bf A3}. Note that $y_0\equiv(\Lambda_0,\Phi_0)$ is a solution of
\eqref{problem-eigen-compact-L} satisfying \eqref{assumption-a2}.

We define bilinear form $a'(\Phi_0;\cdot,\cdot)$ by
\begin{eqnarray*}
a'(\Phi_0;\Psi,\Gamma)=\langle
 \mathcal{L}_{\Phi_0}'(\Lambda_0,\Phi_0)\Psi,\Gamma\rangle \quad\forall~ \Psi,
\Gamma\in \mathcal{H}.
\end{eqnarray*}
Obviously, $a'(\Phi_0;\cdot,\cdot)$ is continuous on
$\mathcal{H}\times \mathcal{H}$.

 Now we shall  introduce the following adjoint problem:  for $f \in
(L^2(\Omega))^N$, find $\Psi_f\in \mathcal{T}_{\Phi_0}$ such that
\begin{eqnarray}\label{problem-adjoint}
a'(\Phi_0;\Psi_f,\Gamma)=(f,\Gamma) \quad \forall~ \Gamma\in
\mathcal{T}_{\Phi_0}.
\end{eqnarray}
Since $\mathcal{L}'_{\Phi_0}(\Lambda_0,\Phi_0)$ is an isomorphism,
\eqref{problem-adjoint} has a unique solution and
\begin{eqnarray}\label{psi}
\|\Psi_f\|_{1,\Omega}\lesssim \|f\|_{0,\Omega}.
\end{eqnarray}
Let $K:((L^2(\Omega))^N, (\cdot,\cdot))\to
(\mathcal{T}_{\Phi_0},(\nabla \cdot,\nabla \cdot))$ be the operator
satisfying
\begin{eqnarray}\label{K-operator}
(\nabla Kw,\nabla v)=(w,v)\qquad \forall~w\in (L^2(\Omega))^N,
\forall~ v\in\mathcal{T}_{\Phi_0}.
\end{eqnarray}
Then $K$ is compact. Set
\begin{eqnarray*}
\rho_n=\sup_{f\in (L^2(\Omega))^N,\|f\|_{0,\Omega}\leq
1}\inf_{\Psi\in\mathcal{H}_n}\|\nabla\big((\mathcal{L}_{\Phi_0}'(\Lambda_0,\Phi_0))^{-1}Kf-\Psi\big)\|_{0,\Omega},
\end{eqnarray*}
we then have the following estimate (see, e.g., \cite{bao-zhou04})
\begin{eqnarray}\label{tmp-2}
\|\nabla\big((\mathcal{L}_{\Phi_0}'(\Lambda_0,\Phi_0))^{-1}Kf-P_n'(\mathcal{L}_{\Phi_0}'(\Lambda_0,\Phi_0))^{-1}Kf\big)\|_{0,\Omega}\lesssim
\rho_n\|f\|_{0,\Omega} \qquad \forall~f\in(L^2(\Omega))^N
\end{eqnarray}
with $$\lim_{n\to \infty}\rho_n=0,$$ where  $P_n':
\mathcal{T}_{\Phi_0}\to \mathcal{T}_{\Phi_0}\cap \mathcal{H}_n$ is
the projection operator satisfying
$$ (\nabla (\Phi_1-P_n'\Phi_1), \nabla \Phi_2)=0 \qquad\forall~ \Phi_2\in\mathcal{T}_{\Phi_0}\cap \mathcal{H}_n .$$

\begin{theorem}\label{priori}
 If
 Assumptions {\bf A2} and {\bf A3} are satisfied, then there exists
$\delta>0$ such that for sufficiently large $n$,
\eqref{problem-eigen-dis} has a unique local solution
$(\Lambda_n,\Phi_n)\in X_{\Phi_0,n}\cap B_{\delta}(y_0)$ satisfying
\begin{eqnarray}\label{H1-norm}
\|\Phi_0-\Phi_n\|_{1,\Omega}\lesssim
d_{\mathcal{H}}(\mathcal{G},\mathcal{H}_n)
\end{eqnarray}
and
\begin{eqnarray}\label{L2-norm}
\|\Phi_0-\Phi_n\|_{0,\Omega}+|\Lambda_0-\Lambda_n|\lesssim \rho_n
\|\Phi_0-\Phi_n\|_{1,\Omega}
\end{eqnarray}
with $\rho_n\to 0$ as $n\to \infty.$
\end{theorem}
\begin{proof}
We obtain from Theorem \ref{unique} that there exists $\delta>0$
such that for sufficiently large $n$, \eqref{problem-eigen-dis} has
a unique local solution $y_n\equiv(\Lambda_n,\Phi_n)\in
X_{\Phi_0,n}\cap B_{\delta}(y_0)$. Hence, we have
\begin{eqnarray*}
F_{n}(y_{n})-F_{n}(\Pi_{n} y_0)=-F_{n}(\Pi_{n} y_0),
\end{eqnarray*}
which leads to
\begin{eqnarray*}
F'_{{n},\Pi_{n} y_0}(y_{n}-\Pi_{n} y_0) =-F_{n}(\Pi_{n} y_0) -
(y_{n}-\Pi_{n} y_0) \int_0^1\big(F'_{{n},\Pi_{n} y_0+t(y_{n}-\Pi_{n}
y_0)}-F'_{n,\Pi_{n} y_0}\big)dt.
\end{eqnarray*}
Using the similar arguments in the proof of Theorem \ref{unique},
 we obtain from Lemma \ref{lemma-h-isomorphism}  that for sufficiently
large $n$
\begin{eqnarray*}\label{proof-33}
\|y_{n}-\Pi_{n} y_0\|_{X_{\Phi_0}} \lesssim \|y_0-\Pi_{n}
y_0\|_{X_{\Phi_0}} +\|y_{n}-\Pi_{n}
y_0\|_{X_{\Phi_0}}(\|y_{n}-\Pi_{n}
y_0\|_{X_{\Phi_0}}^{\alpha}+\|y_{n}-\Pi_{n} y_0\|_{X_{\Phi_0}}^2),
\end{eqnarray*}
which together with \eqref{pi-n-convergence} and the fact that
$y_n\in B_{\delta}(y_0) $ implies that for sufficiently large $n$
\begin{eqnarray}\label{proof-3}
\|y_{n}-\Pi_{n} y_0\|_{X_{\Phi_0}}\lesssim \|y_0-\Pi_{n}
y_0\|_{X_{\Phi_0}}.
\end{eqnarray}
Using \eqref{pi-n-convergence--} and \eqref{proof-3}, we conclude
\begin{eqnarray*}
\|y_{n}-y_0\|_{X_{\Phi_0}} \lesssim \|y_{n}-\Pi_{n}
y_0\|_{X_{\Phi_0}} + \|y_0-\Pi_{n}
y_0\|_{X_{\Phi_0}}\lesssim\inf_{\Psi\in\mathcal{H}_{n}}\|\Psi-\Phi_0\|_{1,\Omega},
\end{eqnarray*}
which implies \eqref{H1-norm}.

Since there exists  $\delta_i\in [0,1]$ such that
\begin{eqnarray*}
(\mathcal{E}'(\rho_{\Phi_n})\phi_{i,n}-\mathcal{E}'(\rho_{\Phi_0})\phi_{0,i},\phi_{j,n})&=&
\int_{\Omega}
(\mathcal{E}'(\rho_{\xi})+2\xi_i^2\mathcal{E}''(\rho_{\xi}))(\phi_{i,n}-\phi_{0,i})\phi_{j,n},
\end{eqnarray*}
where $\xi=(\xi_1,\xi_2,\cdots,\xi_N)$ with
$\xi_i=\delta_i\phi_{i,n}+(1-\delta_i)\phi_{0,i}$, using Assumption
{\bf A2} we get
\begin{eqnarray*}
&
&(\mathcal{E}'(\rho_{\Phi_n})\phi_{i,n}-\mathcal{E}'(\rho_{\Phi_0})\phi_{0,i},\phi_{j,n})
  \lesssim
\int_{\Omega}(\rho_{\xi}+\rho_{\xi}^{\alpha})(\phi_{i,n}-\phi_{0,i})\phi_{j,n}\\
&\lesssim& \|\rho_{\xi}^{\alpha}\|_{0,3/{\alpha},\Omega}
\|\phi_{i,n}-\phi_{0,i}\|_{0,\Omega}
\|\phi_{j,n}\|_{0,6/(3-2\alpha),\Omega}+\|\rho_{\xi}\|_{0,3,\Omega}
\|\phi_{i,n}-\phi_{0,i}\|_{0,\Omega} \|\phi_{j,n}\|_{0,6,\Omega} \\
&\lesssim&  \|\phi_{i,n}-\phi_{0,i}\|_{0,\Omega},
\end{eqnarray*}
from which we have
\begin{eqnarray*}
&
&\big((\mathcal{E}'(\rho_{\Phi_n})-\mathcal{E}'(\rho_{\Phi_0}))\phi_{i,n},\phi_{j,n}\big)\\
&=&(\mathcal{E}'(\rho_{\Phi_n})\phi_{i,n}-\mathcal{E}'(\rho_{\Phi_0})\phi_{0,i},\phi_{j,n})
+(\mathcal{E}'(\rho_{\Phi_0})(\phi_{0,i}-\phi_{i,n}),\phi_{j,n})\nonumber\\
&\lesssim& \|\phi_{i,n}-\phi_{0,i}\|_{0,\Omega}.
\end{eqnarray*}
Note that
\begin{eqnarray*}
&&\lambda_{ij,n}-\lambda_{0,ij}=(A_{\Phi_n}\phi_{i,n},\phi_{j,n})
-(A_{\Phi_0}\phi_{0,i},\phi_{0,j})\nonumber\\
&=&(A_{\Phi_0}(\phi_{i,n}-\phi_{0,i}),\phi_{j,n}-\phi_{0,j})
+\int_{\Omega}\sum_{k=1}^n\lambda_{0,ik}\phi_{0,k}(\phi_{j,n}-\phi_{0,j})\nonumber\\
&
&+\int_{\Omega}\sum_{k=1}^n\lambda_{0,jk}\phi_{0,k}(\phi_{i,n}-\phi_{0,i})
+\int_{\Omega}(\mathcal{E}'(\rho_{\Phi_n})-\mathcal{E}'(\rho_{\Phi_0}))
\phi_{i,n}\phi_{j,n}\\
& &+D(\phi_{i,n}\phi_{j,n},\rho_{\Phi_n}-\rho_{\Phi_0}).
\end{eqnarray*}
Hence we  conclude that
\begin{eqnarray}\label{lambda-pro}
|\Lambda_n-\Lambda_0|\lesssim
\|\Phi_n-\Phi_0\|^2_{1,\Omega}+\|\Phi_n-\Phi_0\|_{0,\Omega}.
\end{eqnarray}
By Lemma \ref{lemma-decomposition}, we decompose $\Phi_n$ as
\begin{eqnarray}\label{decomp}
\Phi_n=\Phi_0+\mathcal{S}(W)\Phi_0+W,
\end{eqnarray}
where  $W\in \mathcal{T}_{\Phi_0}$ and $\mathcal{S}(W)\in
\mathcal{S}^{N\times N}$ satisfying
\begin{eqnarray}\label{property-SW}
|\mathcal{S}(W)|\leq \|W\|_{0,\Omega}^2\leq
\|\Phi_0-\Phi_n\|_{0,\Omega}^2.
\end{eqnarray}
Setting $\Psi=\Psi_{\Phi_n-\Phi_0}$ and applying  the duality
problem of \eqref{problem-adjoint}, we obtain
\begin{eqnarray*}
\|\Phi_n-\Phi_0\|_{0,\Omega}^2&=&(\Phi_n-\Phi_0,\Phi_n-\Phi_0)\nonumber\\
&=&(\Phi_n-\Phi_0,\mathcal{S}(W)\Phi_0)+(\Phi_n-\Phi_0,W)\nonumber\\
&=&(\Phi_n-\Phi_0,\mathcal{S}(W)\Phi_0)+a'(\Phi_0;\Psi,W),
\end{eqnarray*}
which together with \eqref{decomp} leads to
\begin{eqnarray*}
\|\Phi_n-\Phi_0\|_{0,\Omega}^2&=&(\Phi_n-\Phi_0,\mathcal{S}(W)\Phi_0)
-a'(\Phi_0;\Psi,\mathcal{S}(W)\Phi_0)+a'(\Phi_0;\Psi,\Phi_n-\Phi_0)\nonumber\\
&=&(\Phi_n-\Phi_0,\mathcal{S}(W)\Phi_0)-a'(\Phi_0;\Psi,\mathcal{S}(W)\Phi_0)
+a'(\Phi_0;\Psi-P_n'\Psi,\Phi_n-\Phi_0)\nonumber\\
& &+a'(\Phi_0;P_n'\Psi,\Phi_n-\Phi_0).
\end{eqnarray*}
Note that from \eqref{problem-eigen-compact-L} and
\eqref{problem-eigen-dis}, we have
\begin{eqnarray*}
2a'(\Phi_0;P_n'\Psi,\Phi_n-\Phi_0)&=&E''(\Phi_0)(P_n'\Psi,\Phi_n-\Phi_0)
-E'(\Phi_n)(P_n'\Psi)+E'(\Phi_0)(P_n'\Psi)\nonumber\\
& &+2\sum_{i,j=1}^N(\lambda_{ij,n}
-\lambda_{0,ij})\int_{\Omega}\phi_{j,n}P_n'\psi_i
\end{eqnarray*}
while the fact that $\Psi\in \mathcal{T}_{\Phi_0}$ yields
\begin{eqnarray*}
\int_{\Omega}\phi_{j,n}P_n'\psi_i=\int_{\Omega}(\phi_{j,n}
-\phi_{0,j})\psi_i+\int_{\Omega}\phi_{j,n}(P_n'\psi_i-\psi_i),
\end{eqnarray*}
we then come to
\begin{eqnarray*}
\|\Phi_n-\Phi_0\|_{0,\Omega}^2&=&(\Phi_n-\Phi_0,\mathcal{S}(W)\Phi_0)
-a'(\Phi_0;\Psi,\mathcal{S}(W)\Phi_0)
+a'(\Phi_0;\Psi-P_n'\Psi,\Phi_n-\Phi_0)\nonumber\\
&&-\frac{1}{2}\big(E'(\Phi_n)(P_n'\Psi)-E'(\Phi_0)(P_n'\Psi)
-E''(\Phi_0)(P_n'\Psi,\Phi_n-\Phi_0)\big)\nonumber\\
& &+\sum_{i,j=1}^N(\lambda_{ij,n}-\lambda_{0,ij})
\big(\int_{\Omega}(\phi_{j,n}-\phi_{0,j})\psi_i
+\int_{\Omega}\phi_{j,n}(P_n'\psi_i-\psi_i)\big).
\end{eqnarray*}
Using the  Taylor expansion, we have that there exists
  $\delta \in [0,1]$ such that
\begin{eqnarray}\label{temp}
&&E'(\Phi_n)(P_n'\Psi)-E'(\Phi_0)(P_n'\Psi)
-E''(\Phi_0)(P_n'\Psi,\Phi_n-\Phi_0)\nonumber\\
&=&E''(\xi)(P_n'\Psi,\Phi_n-\Phi_0)-E''(\Phi_0)(P_n'\Psi,\Phi_n-\Phi_0)\nonumber\\
&\lesssim&
(\|\Phi_n-\Phi_0\|_{1,\Omega}^{\alpha}+\|\Phi_n-\Phi_0\|_{1,\Omega}^{2})\|\Phi_n-\Phi_0\|_{0,\Omega}^2,
\end{eqnarray}
where $\xi=\Phi_0+\delta (\Phi_n-\Phi_0)$ and the last inequality is
obtained by the similar arguments in the proof of \eqref{tmp-1} or
Lemma 4.5 in \cite{cancesM10} when $\Gamma_1=\Phi_n-\Phi_0$,
$\Gamma_2=\Phi_n-\Phi_0$ and $\Gamma_3=P_n'\Psi$, and using the fact
\begin{eqnarray*}
\|P_n'\Psi\|_{1,\Omega}\lesssim \|\Psi\|_{1,\Omega}\lesssim
\|\Phi_n-\Phi_0\|_{0,\Omega}.
\end{eqnarray*}

Taking   \eqref{psi}, \eqref{tmp-2}, \eqref{property-SW} and
\eqref{temp} into account, we obtain that
\begin{eqnarray*}
\|\Phi_n-\Phi_0\|_{0,\Omega}&\lesssim&
\|\Phi_n-\Phi_0\|_{0,\Omega}^2+\rho_n\|\Phi_n-\Phi_0\|_{1,\Omega}
+\|\Phi_n-\Phi_0\|_{1,\Omega}^{\alpha}\|\Phi_n-\Phi_0\|_{0,\Omega}\nonumber\\
&+&|\Lambda_n-\Lambda_0|(\|\Phi_n-\Phi_0\|_{0,\Omega}+\rho_n),
\end{eqnarray*}
which together with \eqref{lambda-pro} and Theorem \ref{theo-convergence} produces
\begin{eqnarray*}
\|\Phi_n-\Phi_0\|_{0,\Omega}\lesssim
\rho_n\|\Phi_n-\Phi_0\|_{1,\Omega}
\end{eqnarray*}
when $n\gg 1.$ This completes the proof.
\end{proof}
\begin{remark}
Theorem \ref{priori} shows that under certain assumptions every
ground state solution can be approximated with some convergent rate
by finite dimensional solutions. We see that \cite{cancesM10}
provided numerical analysis of plane wave approximations only  while
our results apply to general finite dimensional discretizations and
the analysis is systematic and carried out under very mild
assumptions.
\end{remark}

\begin{remark}\label{quadratic}
If in addition, $V_{loc}\in H^1(\Omega)$, $\zeta_j\in
H^1(\Omega)~(j=1,2,\cdots,M)$ and $\mathcal{E}\in
C^1([0,\infty))\cap C^3((0,\infty))$,
 then for sufficiently large $n$, estimates
\eqref{H1-norm} and \eqref{L2-norm} are also satisfied with
$\tilde{\rho}_n\to 0$ as $n\to\infty$. Here\begin{eqnarray*}
\tilde{\rho}_n=\sup_{f\in \mathcal{H},\|f\|_{1,\Omega}\leq
1}\inf_{\Psi\in\mathcal{H}_n}\|\nabla\big((\mathcal{L}_{\Phi_0}'(\Lambda_0,\Phi_0))^{-1}Kf-\Psi\big)\|_{0,\Omega}
\end{eqnarray*}
and  $K:(\mathcal{H}, (\nabla \cdot,\nabla \cdot))\to
(\mathcal{T}_{\Phi_0},(\nabla \cdot,\nabla \cdot))$ satisfying
\eqref{K-operator}.
\end{remark}

\begin{remark}\label{finite-result}
Let $y_0\equiv(\Lambda_0,\Phi_0)$ be the ground state solution of
\eqref{problem-eigen-compact-L} satisfying \eqref{assumption-a2}. We
assume that $\Omega$ is a convex bounded domain and $S_n$ is
replaced by the standard  finite element space $S_0^{h,k}(\Omega)$
of piecewise polynomials of degree $k~ (k=1,2)$ of $H^1_0(\Omega)$
over a shape-regular mesh with size $h$. Let
$(\Lambda_{h,k},\Phi_{h,k})\in X_{\Phi_0,h}$ be the ground state
solution of \eqref{problem-eigen-dis}  and Assumption {\bf A2} hold.
Then
\begin{eqnarray*}
|\Lambda_0-\Lambda_{h,1}|+\|\Phi_0-\Phi_{h,1}\|_{0,\Omega}+h\|\Phi_0-\Phi_{h,1}\|_{1,\Omega}\lesssim
h^2
\end{eqnarray*}
when $h\ll 1$. If in addition, $V_{loc}\in H^1(\Omega)$, $\zeta_j\in
H^1(\Omega)~(j=1,2,\cdots,M)$ and $\mathcal{E}\in
C^1([0,\infty))\cap C^3((0,\infty))$, then
\begin{eqnarray*}
|\Lambda_0-\Lambda_{h,2}|+h\|\Phi_0-\Phi_{h,2}\|_{0,\Omega}+h^2\|\Phi_0-\Phi_{h,2}\|_{1,\Omega}\lesssim
h^4
\end{eqnarray*}
when $h\ll 1$.
\end{remark}

\section{Numerical examples}\label{sec-numerical}
\setcounter{equation}{0}

In this section, we will report several numerical examples that
support our theory. These numerical experiments were carried out on
LSSC3 cluster in the State Key Laboratory of Scientific and
Engineering Computing, Chinese Academy of Sciences. Our code is
based on the PHG finite element toolbox developed in the State Key
Laboratory of Scientific and Engineering Computing, Chinese Academy
of Sciences.

In these examples, we solved  Kohn-Sham equation (\ref{problem-eigen-compact-L}). We chose our computational domain $\Omega$ as $[-10.0,10.0]^3$.
 In computation, we used the norm-conserving
pseudopotential \cite{troullier90} obtained by fhi98PP software and applied the local density approximation (LDA) for the exchange-correction potential.
We applied the standard linear and quadratic finite element discretizations over uniform tetrahedral triangulations. The finite dimensional nonlinear
eigenvalue problems were then solved by self consistent field iterations.
In each iteration, the Kohn-Sham Hamiltonian is constructed from a trial electron density, the electron density is then obtained from the low-lying
eigenfunctions of the discretized Hamiltonian, the resulting electron density and the trial electron density are then mixed and form a new trial electron
density. The loop continues until self-consistency of the electron density is reached.

We present numerical results for $N_2$, $C_2H_4$ and $SiH_4$
molecules. Since analytical solutions are not available, we use the
numerical solutions on a very fine grid for references to calculate
the approximation errors.

Let us first come to the ground state total energy approximations. The errors of total energy of $N_2$, $C_2H_4$ and $SiH_4$ are presented in Figures
\ref{fig:n2-energy}, \ref{fig:c2h4-energy} and \ref{fig:sih4-energy}, respectively. We can see that  convergence rates for linear and quadratic finite
elements are $h^2$ and $h^4$ respectively, which agrees well with the results predicted by Theorem \ref{thm:error-energy}.  We then present the
approximation errors of the first two eigenvalues for these three molecules, see Figures \ref{fig:n2-eval1}, \ref{fig:c2h4-eval1} and
\ref{fig:sih4-eval1}. We may see that these results coincide well with our theory (see, e.g., Remark \ref{finite-result}), too.

\begin{figure}[ht]
  \centering
  \subfloat[Linear finite elements]{
    \includegraphics[width=7.0cm]{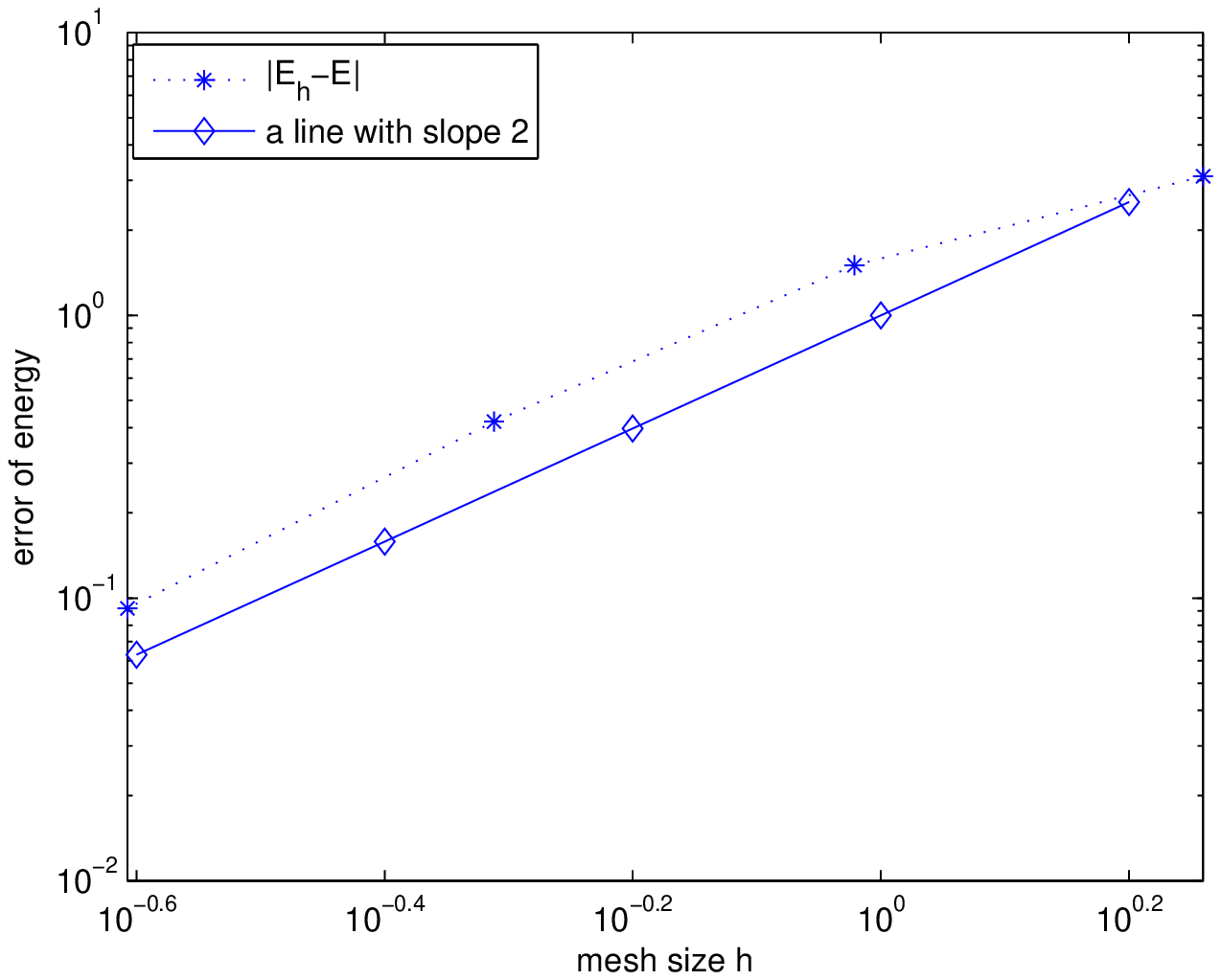}
  }
  \,
  \subfloat[Quadratic finite elements]{
    \includegraphics[width=7.0cm]{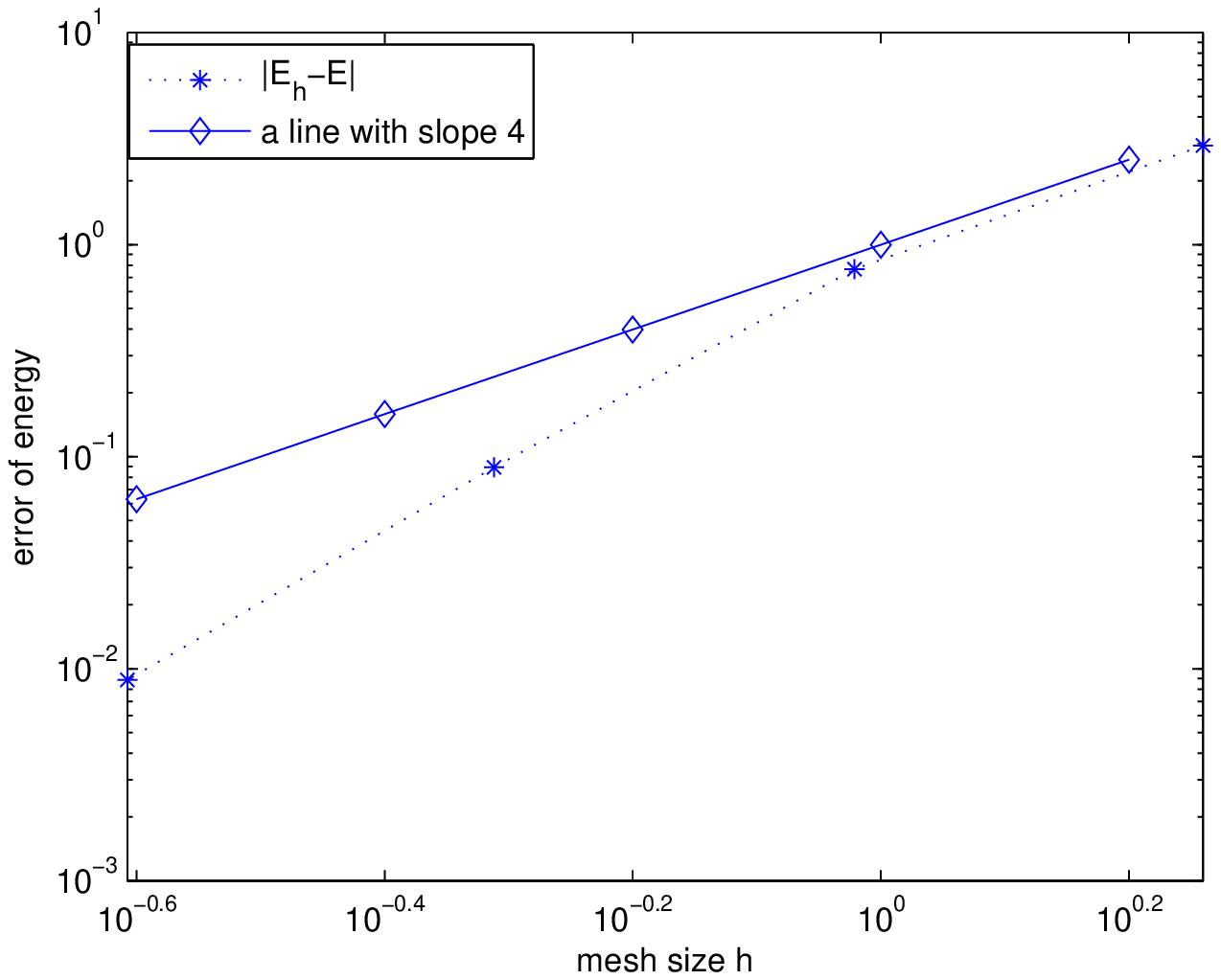}
  }
  \caption{$N_2$: errors of the ground state total energy}
  \label{fig:n2-energy}
\end{figure}

\begin{figure}[htb]
  \centering
  \subfloat[Linear finite elements]{
    \includegraphics[width=7.0cm]{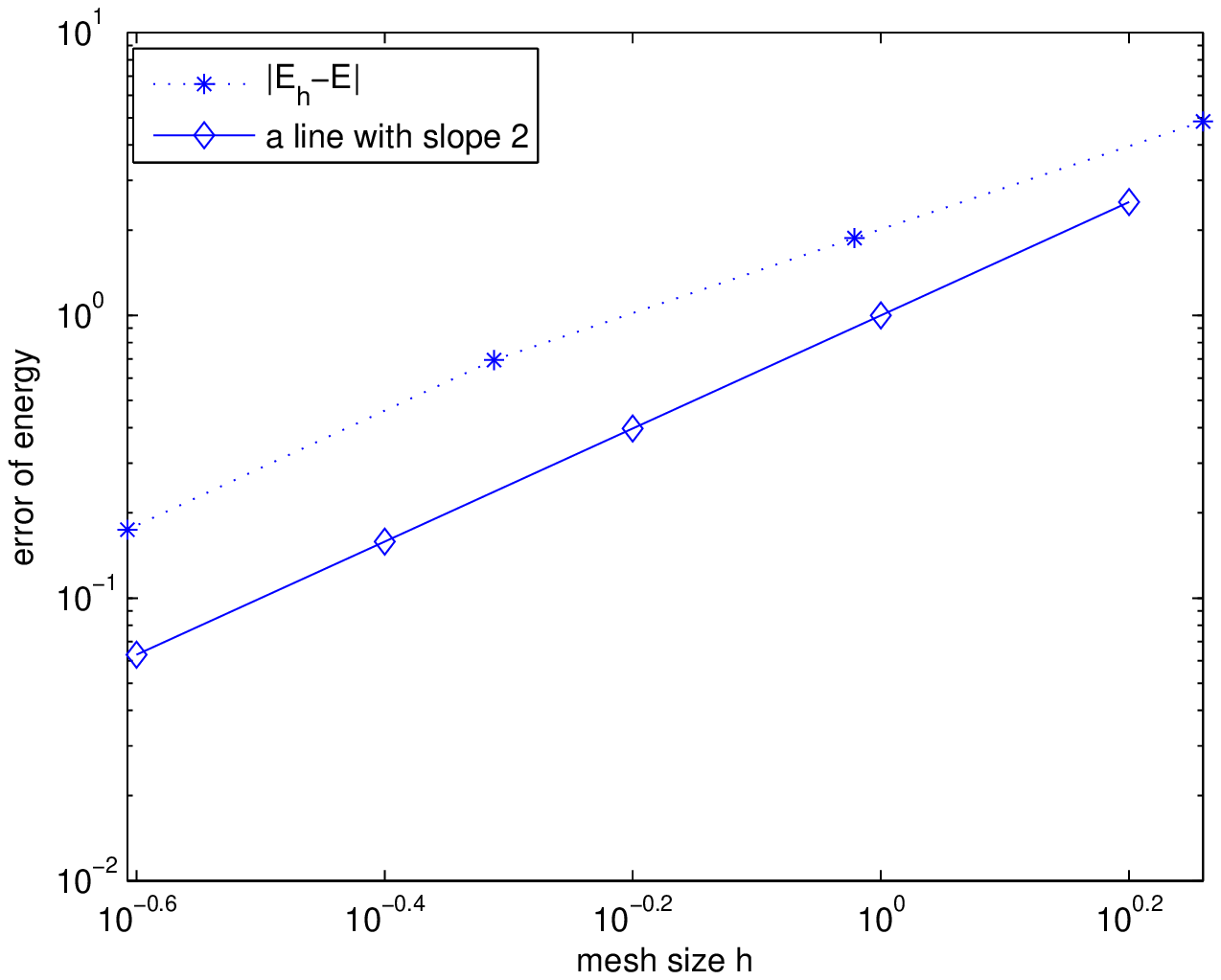}
  }
  \,
  \subfloat[Quadratic finite elements]{
    \includegraphics[width=7.0cm]{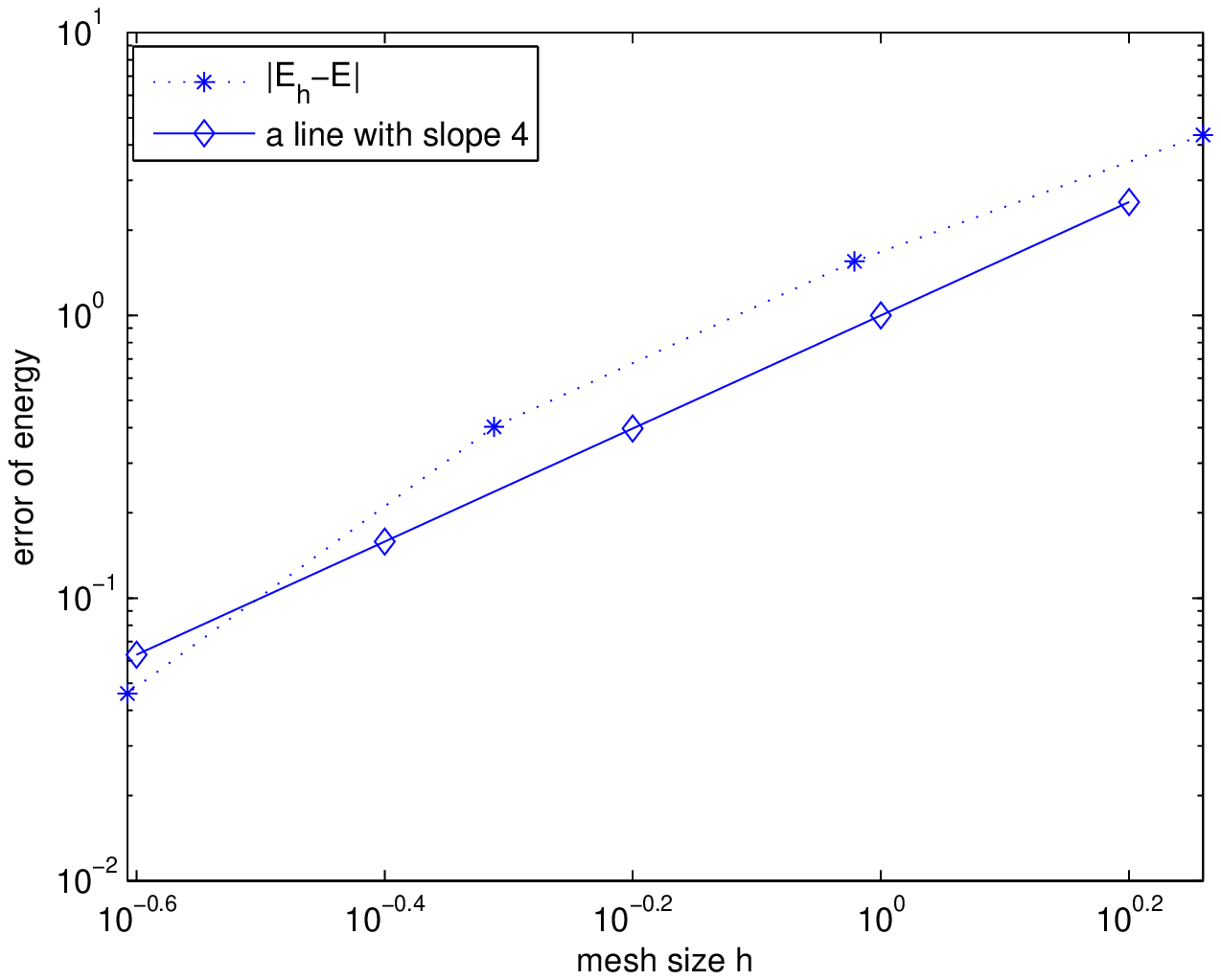}
  }
  \caption{$C_2H_4$: errors of the ground state total energy}
  \label{fig:c2h4-energy}
\end{figure}

\begin{figure}[htbp]
  \centering
  \subfloat[Linear finite elements]{
    \includegraphics[width=7.0cm]{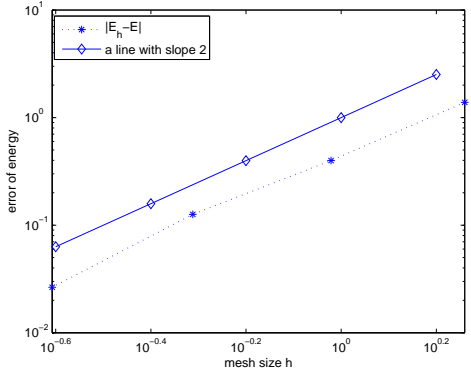}
  }
  \,
  \subfloat[Quadratic finite elements]{
    \includegraphics[width=7.0cm]{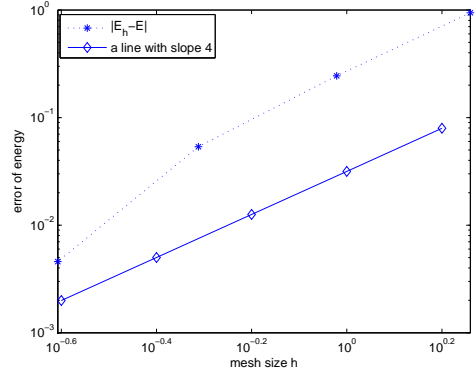}
  }
  \caption{$SiH_4$: errors of the ground state total energy}
  \label{fig:sih4-energy}
\end{figure}

\begin{figure}[ht]
  \centering
  \subfloat[Linear finite elements]{
    \includegraphics[width=7.0cm]{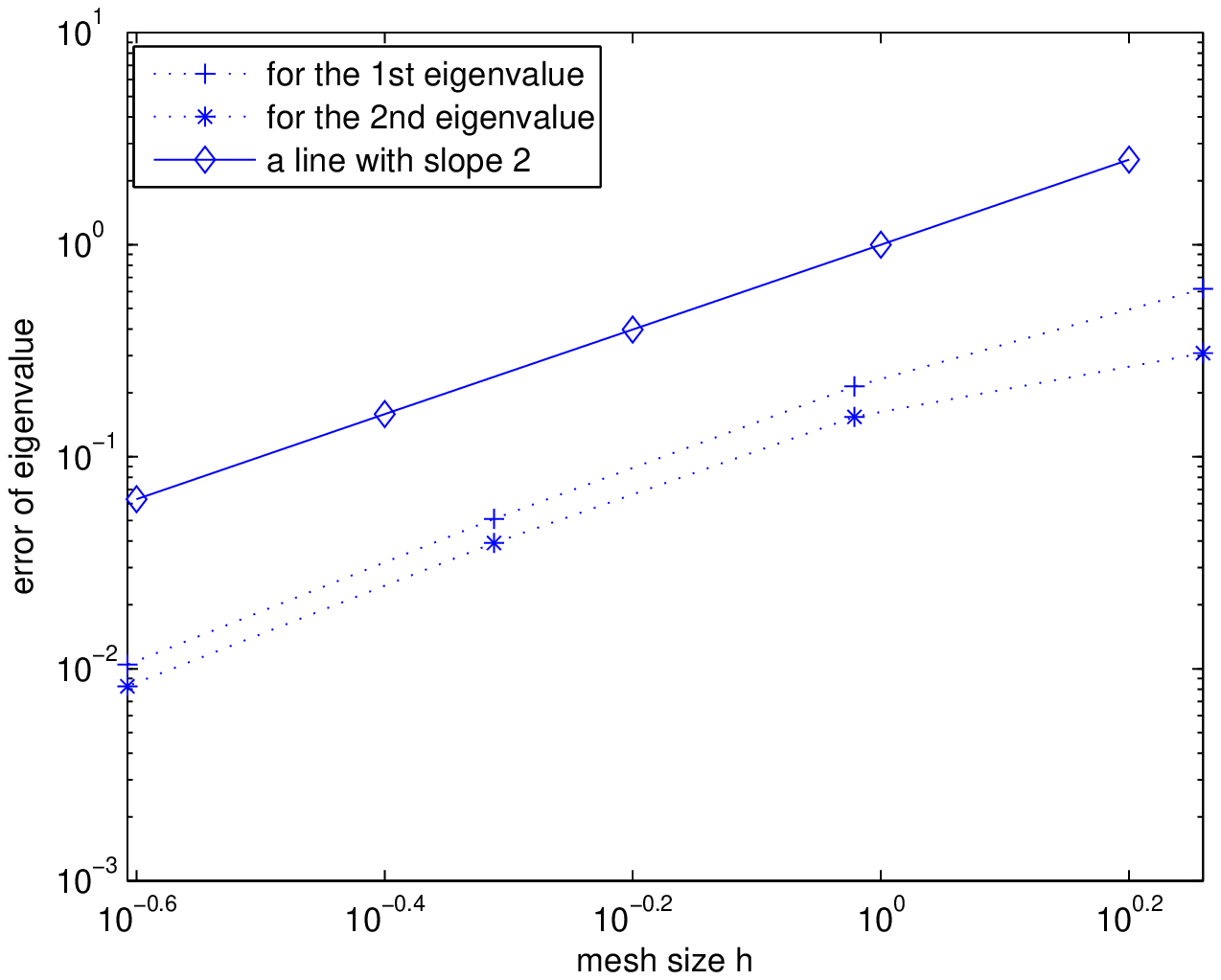}
  }
  \,
  \subfloat[Quadratic finite elements]{
    \includegraphics[width=7.0cm]{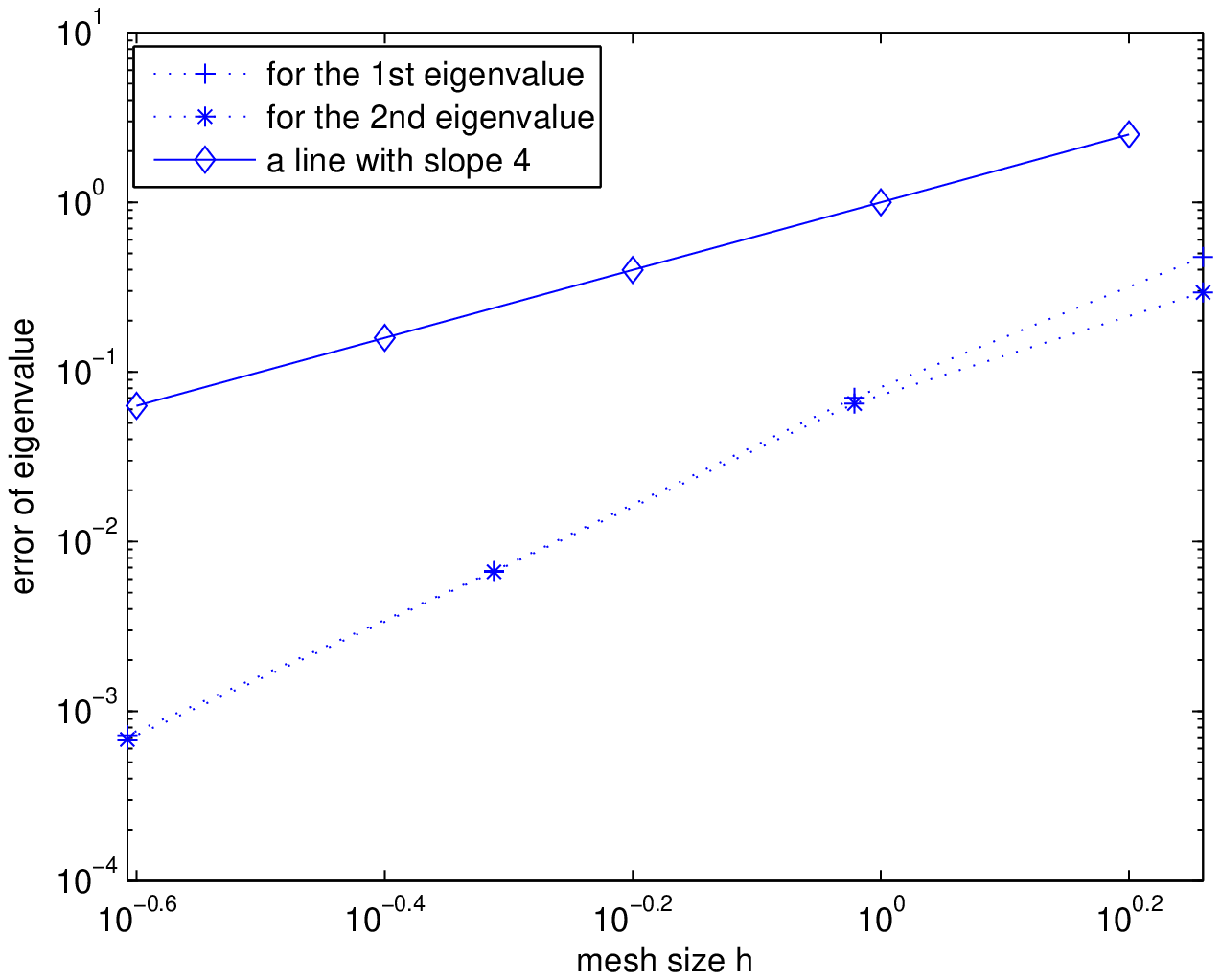}
  }
  \caption{$N_2$: errors of the first  and second eigenvalues}
  \label{fig:n2-eval1}
\end{figure}

\begin{figure}[htb]
  \centering
  \subfloat[Linear finite elements]{
    \includegraphics[width=7.0cm]{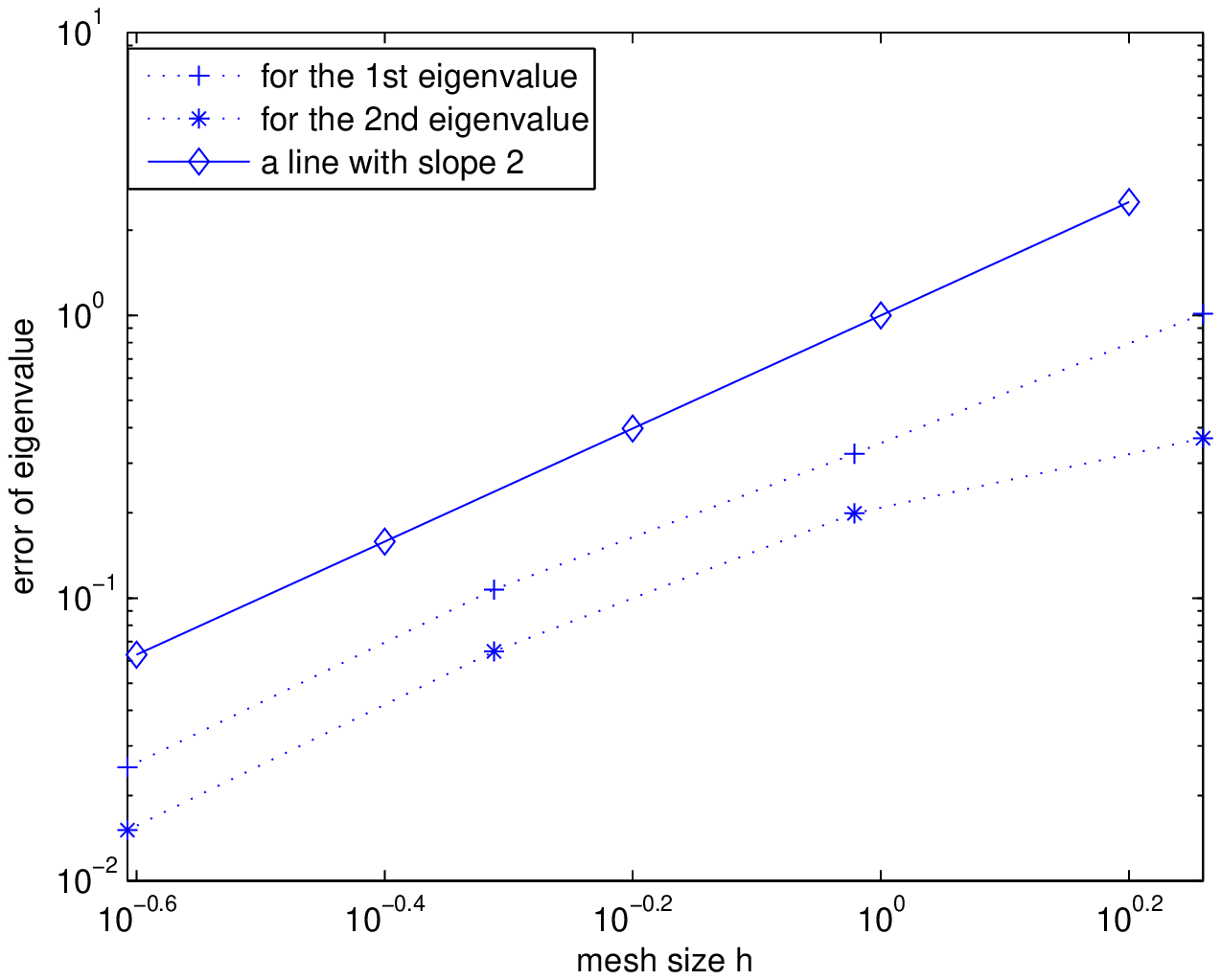}
  }
  \,
  \subfloat[Quadratic finite elements]{
    \includegraphics[width=7.0cm]{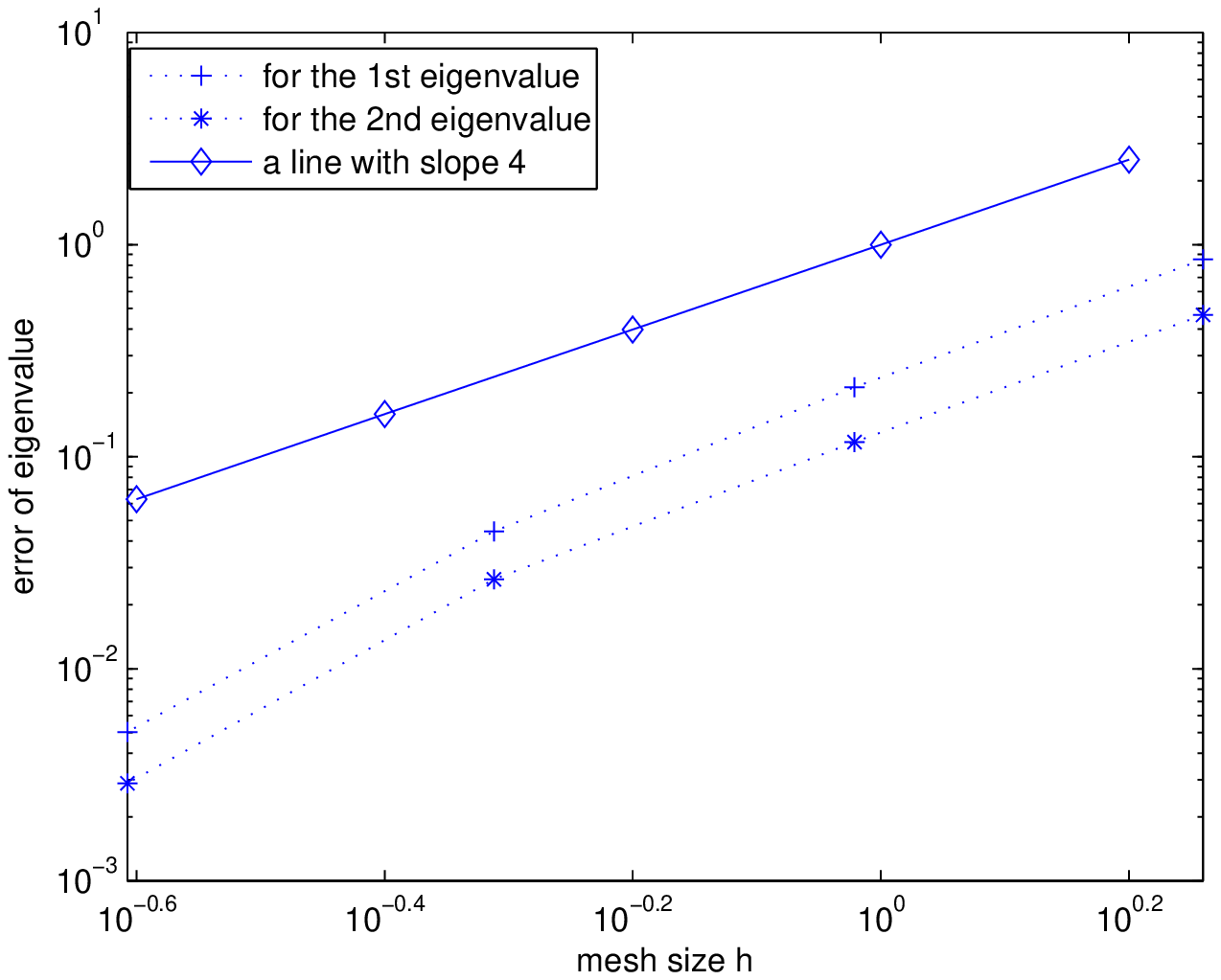}
  }
  \caption{$C_2H_4$: errors of the first and second eigenvalues}
  \label{fig:c2h4-eval1}
\end{figure}

\begin{figure}[htbp]
  \centering
  \subfloat[Linear finite elements]{
    \includegraphics[width=7.0cm]{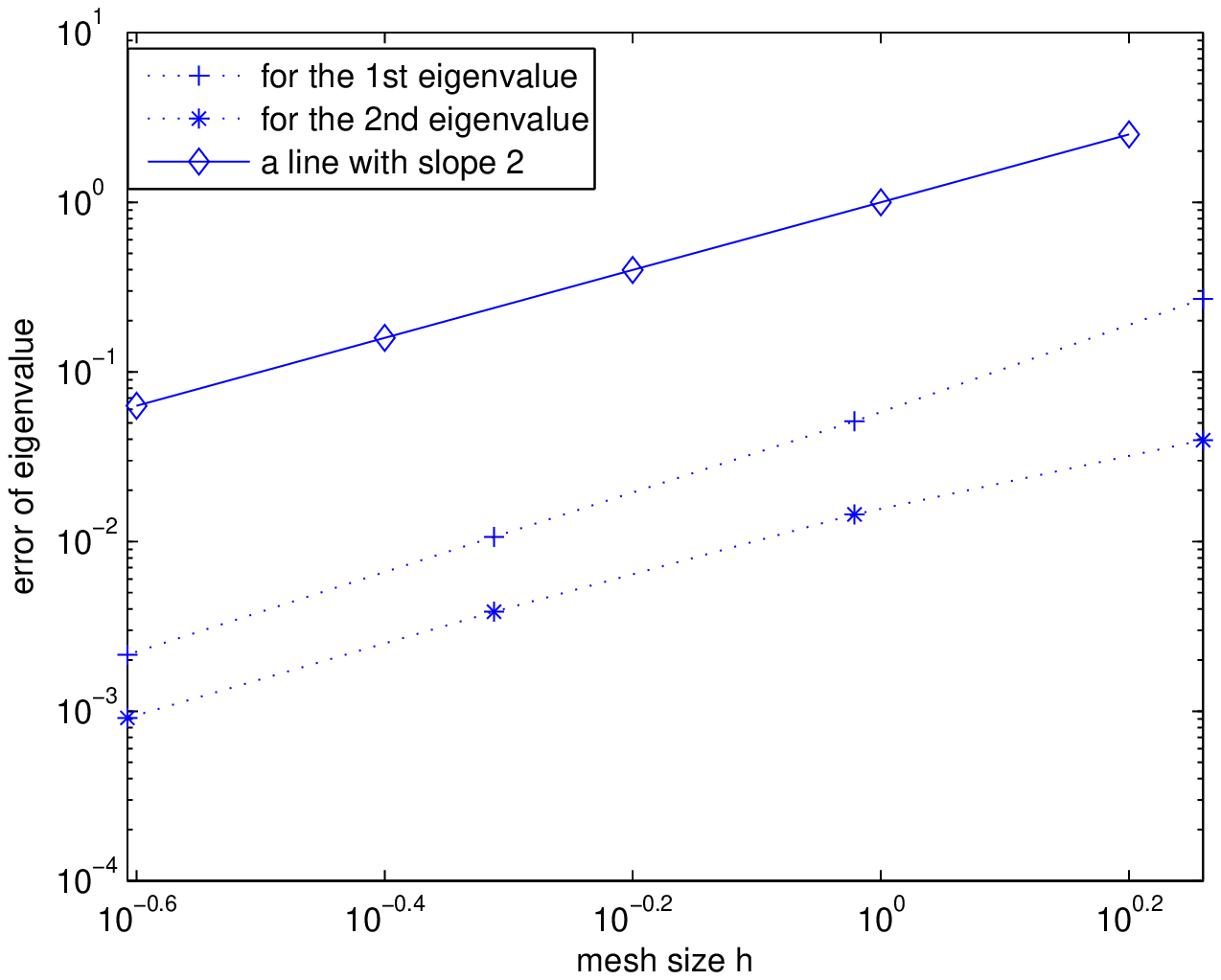}
  }
  \,
  \subfloat[Quadratic finite elements]{
    \includegraphics[width=7.0cm]{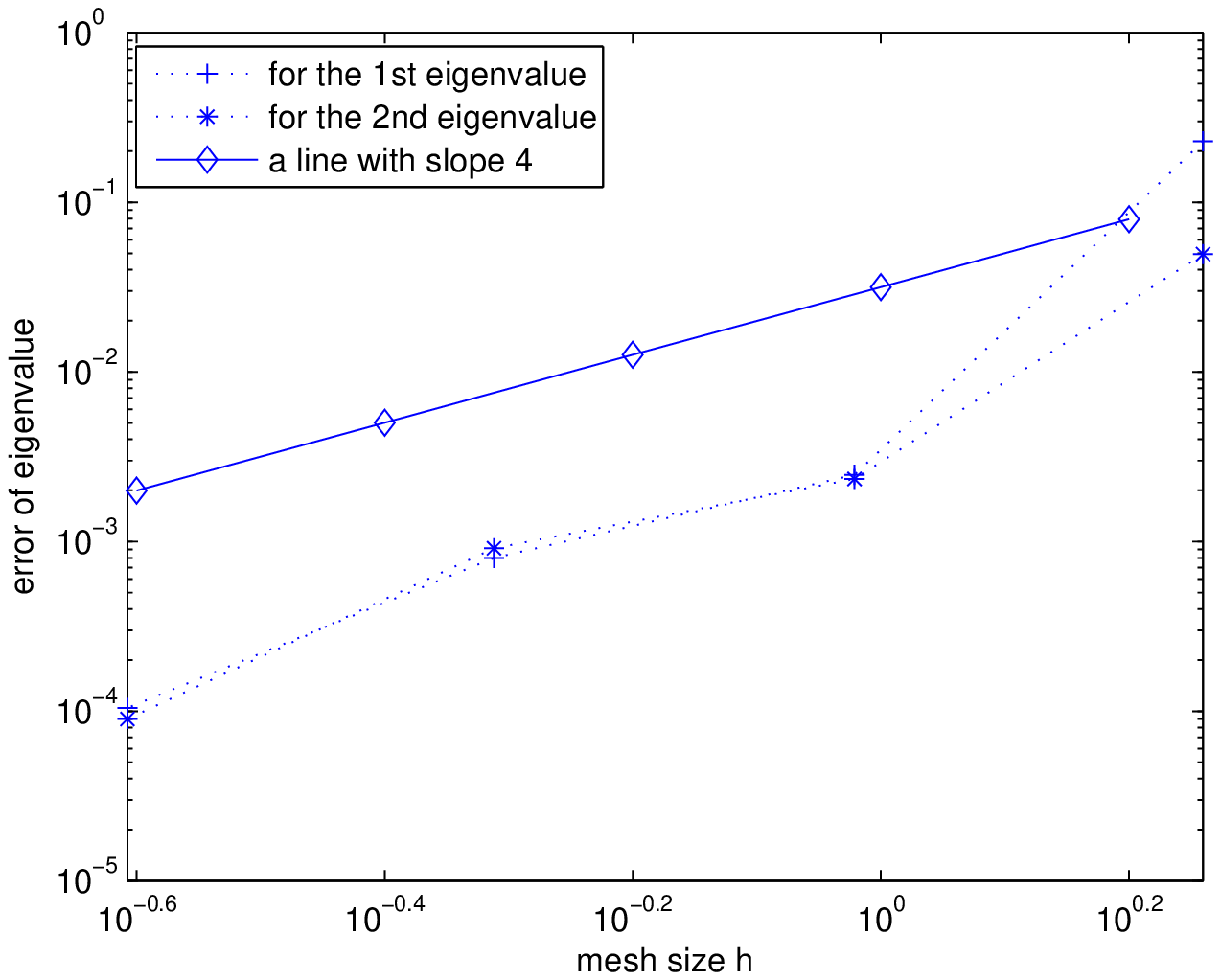}
  }
  \caption{$SiH_4$: errors of the first and second eigenvalues}
  \label{fig:sih4-eval1}
\end{figure}

\section{Concluding remarks}\label{sec-con} \setcounter{equation}{0}
We have analyzed finite dimensional approximations of Kohn-Sham
models. We have proved the convergence and shown the optimal a
priori error estimates of finite dimensional approximations.

As we see, the ground state solutions oscillate  near the nuclei \cite{gong-shen-zhang-zhou-08,kato57}. It is natural to apply adaptive finite element
discretizations to carry out the electronic structure calculations. Indeed, it is our on-going work to study the convergence and complexity of adaptive
finite element methods that will be addressed elsewhere.\vskip 0.2cm

{\sc Acknowledgements.} The authors would like to thank Dr. Xiaoying
Dai, Prof. Lihua Shen, and Dr. Dier Zhang for their stimulating
discussions and fruitful cooperations on electronic structure
computations that have motivated this work. The authors are grateful
to Prof. Linbo Zhang and Dr. Tao Cui for their assistance on
numerical computations and, to Mr. Zaikun Zhang for his discussions
on the local uniqueness of the discrete solution.


\begin{thebibliography}{99}

\bibitem{agmon81} S. Agmon, {\em Lectures on the Exponential Decay
of Solutions of Second-Order Elliptic Operators}, Princeton
University Press, Princeton, 1981.

\bibitem{anantharaman09} A. Anantharaman and E. Canc\`{e}s, {\em
Existence of minimizers for Kohn-Sham models in quantum chemistry},
Ann. I. H. Poincar\'{e}-AN, {\bf 26} (2009), pp.~2425-2455.

\bibitem{arias99} T.A. Arias, {\em Multiresolution analysis of electronic structure:
Semicardinal and wavelet bases},  Rev. Mod. Phys., {\bf 71} (1999),
pp.~267-311.

\bibitem{bao-zhou04} G. Bao and A. Zhou, {\em Analysis of finite dimensioanl approximations to
a class of partial differential equaitons}, Math. Meth. Appl. Sci.,
{\bf 27} (2004), pp.~2055-2066.


\bibitem{cancesM09} E. Canc\`{e}s, R. Chakir, and Y. Maday, {\em
Numerical analysis of nonlinear eigenvalue problems}, J. Sci.
Comput., {\bf 45} (2010), pp.~90-117.


\bibitem{cancesM10} E. Canc\`{e}s, R. Chakir, and Y. Maday, {\em
Numerical analysis of the planewave discretization of some
orbital-free and Kohn-Sham models}, arXiv:1003.1612,  2010.

\bibitem{canuto07} C. Canuto, M.Y. Hussaini, A. Quarteroni,
and T.A. Zang, {\em Spectral Methods}, Springer-Verlag, Berlin
Heidelberg, 2007.


\bibitem{chen-gong09} H. Chen, X. Gong, and A. Zhou, {\em
Numerical approximations of a nonlinear eigenvalue problem and
applications to a density functional model}, Math. Meth. Appl. Sci.,
{\bf 33} (2010), pp.~1723-1742.


\bibitem{chen-he-zhou10} H. Chen, L. He, and A. Zhou, {\em Finite
element approximations of nonlinear eigenvalue problems in quantum physics}, Comput. Methods Appl. Mech. Engrg., {\bf 200} (2011), pp. 1846-1865.

\bibitem{chen-zhou08} H. Chen and A. Zhou, {\em Orbital-free
density functional theory for molecular structure calculations},
Numer. Math. Theor. Meth. Appl., {\bf 1} (2008), pp.~1-28.

\bibitem{ciarlet78} P.G. Ciarlet, {\em The Finite Element Method
for Elliptic Problems}, North-Holland, 1978.

\bibitem{edelman98} A. Edelman, T.A. Arias, and S.T. Smith,
{\em The geometry of algorithms with orthogonality constraints},
SIAM J. Matrix Anal. Appl., {\bf 20} (1998), pp.~303-353.

\bibitem{genovese08} L. Genovese, A. Neelov, S. Goedecker, T. Deutsch, S.A.
Ghasemi, A. Willand, D. Caliste, O. Zilberberg, M. Rayson, A.
Bergman, and R. Schneider, {\em Daubechies wavelets as a basis set
for density functional pseudopotential calculations},  J. Chem.
Phys., {\bf 129} (2008), pp.~014109-014112.

\bibitem{gong-shen-zhang-zhou-08} X. Gong, L. Shen, D. Zhang, and A. Zhou,
{\em Finite element approximations for Schr\"{o}dinger equations with applications to electronic structure computations}, J. Comput. Math., {\bf 26}
(2008), pp. 310-323.


\bibitem{hoffmann01} M. Hoffmann-Ostenhof, T. Hoffmann-Ostenhof, and T. \O{}stergaard S\o{}rensen, {\em Electron
wavefunctions and densities for atoms}, Annales Henri Poincar\'{e}, {\bf 2} (2001), pp. 77-100.

\bibitem{hohenberg-kohn64} P. Hohenberg and W. Kohn, {\em Inhomogeneous Electron
Gas}, Phys. Rev. B, {\bf 136} (1964), pp. 864-871.

\bibitem{kato57} T. Kato, {\em On the eigenfunctions of many-particle systems in quantum mechanics},Comm. Pure Appl. Math., {\bf 10} (1957), pp. 151-177.

\bibitem{kohn-sham65} W. Kohn and L.J. Sham, {\em Self-consistent equations including exchange and correlation
effects}, Phys. Rev. A, {\bf 140} (1965), pp. 1133-1138.

\bibitem{langwallner-ortner09} B. Langwallner, C. Ortner, and E. S{\" u}li,
{\em Existence and convergence results for the Galerkin
approximation of an electroniic density functional}, $M^3AS$, doi:
10.1142/S021820251000491X.


\bibitem{bris93} C. Le Bris, {\em Quelques probl\`{e}mes
math\'{e}matiques en chimie quanntique mol\'{e}culaire}, PhD thesis,
\`{E}cole Polytechnique, 1993.

\bibitem{bris03} C. Le Bris, ed., {\em Handbook of Numerical
Analysis, Vol. X. Special issue: Computational Chemistry},
North-Holland, Amsterdam, 2003.


\bibitem{maday00} Y. Maday and G. Turinici, {\em Error bars and quadratically convergent
methods for the numerical simulation of the Hartree-Fock equations},
Numer. Math., {\bf 94} (2000), pp. 739-770.

\bibitem{martin04} R.M. Martin, {\em Electronic Structure:
Basic Theory and Practical Method}, Cambridge University Press,
Cambridge, 2004.

\bibitem{payne92} M.C. Payne, M.P. Teter, D.C. Allan, T.A. Arias, and J.D. Joannopoulos,
{\em Iterative minimization techniques for ab-initio total-energy calculations: Molecular dynamics and conjugategradients},
Rev. Mod. Phys., {\bf 64} (1992), pp.~1045-1097.

\bibitem{parr-yang94} R.G. Parr and W.T. Yang,
{\em Density-Functional Theory of Atoms and Molecules},  Clarendon
Press, Oxford, 1994.


\bibitem{saad10} Y. Saad, J.R. Chelikowsky, and S.M. Shontz,
 {\em Numerical methods for electronic structure calculations of
 materials}, SIAM Review, {\bf 52} (2010), pp.~3-54.

\bibitem{schneider09} R. Schneider, T. Rohwedder, A. Neelov, and
J. Blauert, {\em Direct minimization for calculating invariant
subspaces in density functional computations of the electronic
structure}, J. Comput. Math., {\bf 27} (2009), pp.~360-387.

\bibitem{simon00} B. Simon, {\em Schr{\"{o}}dinger operators in the
twentieth century}, J. Math. Phys., {\bf 41} (2000), pp.~3523-3555.

\bibitem{sur10} P. Suryanarayana, V. Gavini, T. Blesgen,
K. Bhattacharya, and M. Ortiz, {\em Non-periodic finite-element
formulation of Kohn-Sham density functional theory}, J. Mech. Phys.
Solid., {\bf 58} (2010), pp.~256-280.

\bibitem{troullier90} N. Troullier and J.L. Martins, {\em A straightforward method for
generating soft transferable pseudopotentials}, Solid State Comm.,
{\bf 74} (1990), pp.~613-616.

\bibitem{wang-carter00} Y.A.~Wang and E.A.~Carter,
{\em Orbital-free kinetic-energy density functional theory}, in:
Theoretical Methods in Condensed Phase Chemistry (S.~D.~Schwartz,
ed.), Kluwer, Dordrecht, 2000, pp.~117-184.

\bibitem{zhou04} A. Zhou, {\em An analysis of finite-dimensional
approximations for the ground state solution of Bose-Einstein
condensates}, Nonlinearity, {\bf 17} (2004), pp.~541-550.

\bibitem{zhou07} A. Zhou, {\em Finite dimensional approximations
for the electronic ground state solution of a molecular system},
Math. Meth. Appl. Sci., {\bf 30} (2007), pp.~429-447.

\bibitem{zhou10} A. Zhou,
{\em Multi-level adaptive corrections in finite dimensional
approximations}, J. Comput. Math., {\bf 28} (2010), pp. 45-54.

\end{thebibliography}
\end{document}